\documentclass[twoside]{article}

\usepackage{a4wide, amsmath, amsthm, amssymb, amscd, mathptmx}

\usepackage{amsfonts}

\newcommand\imCMsym[4][\mathord]{%
	\DeclareFontFamily{U} {#2}{}
	\DeclareFontShape{U}{#2}{m}{n}{
		<-6> #25
		<6-7> #26
		<7-8> #27
		<8-9> #28
		<9-10> #29
		<10-12> #210
		<12-> #212}{}
	\DeclareSymbolFont{CM#2} {U} {#2}{m}{n}
	\DeclareMathSymbol{#4}{#1}{CM#2}{#3}
}

\imCMsym{cmmi}{124}{\Cjmath}

\usepackage{color}
\usepackage{pstricks}

\usepackage[all]{xy}\CompileMatrices\SelectTips{cm}{12}

\theoremstyle{plain}
\newtheorem{Thm}{\sc Theorem}[section]
\newtheorem{Theorem}[Thm]{\sc Theorem}
\newtheorem{Corollary}[Thm]{\sc Corollary}

\newtheorem*{Corollary*}{\sc Corollary}

\newtheorem{Proposition}[Thm]{\sc Proposition}
\newtheorem*{Proposition*}{\sc Proposition}
\newtheorem{Lemma}[Thm]{\sc Lemma}

\theoremstyle{definition}
\newtheorem{Definition}[Thm]{Definition}

\theoremstyle{remark}
\newtheorem{Remark}[Thm]{Remark}
\newtheorem{Example}[Thm]{Example}
\newtheorem*{Example*}{Example}
\newtheorem*{Remark*}{Remark}

\newcommand{\CC}{{\mathbb C}}

\newcommand{\FF}{{\mathbb F}}

\newcommand{\ZZ}{{\mathbb Z}}

\newcommand{\PP}{{\mathbb P}}
\newcommand{\QQ}{{\mathbb Q}}
\newcommand{\RR}{{\mathbb R}}

\newcommand{\cX}{{\mathcal X}}

\newcommand{\cS}{{\mathcal S}}
\newcommand{\cE}{{\mathcal E}}
\newcommand{\cF}{{\mathcal F}}
\newcommand{\cG}{{\mathcal G}}

\newcommand{\cO}{{\mathcal O}}

\newcommand{\codim}{{\mathop{\rm codim \, }}}
\newcommand{\Pic}{{\mathop{\rm Pic \, }}}
\newcommand{\GL}{\mathop{\rm GL\, }}

\newcommand{\Coh}{{\mathop{\operatorname{Coh}\, }}}
\renewcommand{\Ref}{{\mathop{\operatorname{Ref}\, }}}

\newcommand{\PGL}{\mathop{\rm PGL\, }}

\newcommand{\an}{{\mathop{\rm an }}}

\newcommand{\et}{{\mathop{\rm \acute{e}t \, }}}
\newcommand{\ch}{{\mathop{\rm ch \, }}}
\newcommand{\cHom}{{\mathop{{\cal H}om}}}

\newcommand{\Gr}{{\mathop{Gr}}}
\newcommand{\id}{\mathop{\rm Id}}

\newcommand{\snf}{{\mathop{\rm snf}}}

\newcommand{\nf}{{\mathop{\rm nf}}}

\newcommand{\im}{\mathop{\rm im \, }}
\newcommand{\Ext}{{\mathop{{\rm Ext \,}}}}

\newcommand{\End}{{\mathop{{\mathcal E}nd}\,}}

\newcommand{\cExt}{{\mathop{{\mathcal E}xt}}}

\newcommand{\Tor}{{\mathop{{\mathcal T}or}}}
\newcommand{\coker}{{\mathop{\rm coker \, }}}

\newcommand{\Res}{{\mathop{\rm Res}}}

\newcommand{\reduced}{{\mathop{\rm red}}}

\newcommand{\rk}{{\mathop{\rm rk \,}}}

\newcommand{\Sym}{{\mathop{{\rm Sym \, }}}}

\newcommand{\Supp}{{\mathop{{\rm Supp \,}}}}
\newcommand{\Spec}{{\mathop{{\rm Spec\, }}}}

\newcommand{\Vect}{{\mathop{{\rm Vect \,}}}}

\begin{document}

\markboth {\rm }{}

\title{Simpson's correspondence on singular varieties in positive characteristic}
\author{Adrian Langer} \date{September 19, 2026}

%%%
\maketitle
%%%

%%%%%%%%%%%%%%%%%%%%%%%%%%%%%%%%%%

{\footnotesize
{\noindent \sc Address:}\\
Institute of Mathematics, University of Warsaw,
ul.\ Banacha 2, 02-097 Warszawa, Poland\\
e-mail: {\tt alan@mimuw.edu.pl}
}

\medskip

\begin{abstract}
The main aim of the paper is to provide analogues of Simpson's correspondence on singular projective varieties defined over an algebraically closed field of characteristic $p>0$. There are two main cases. 

In the first case, we consider analogues of numerically flat vector bundles  on a big open subset of a normal projective variety (with arbitrary singularities). 
Here we introduce the S-fundamental group scheme for quasi-projective varieties that admit compactifications with complement of codimension $\ge 2$. For smooth varieties, we prove that this group scheme coincides with the S-funda\-men\-tal group scheme of the regular locus of any of its (small) projective compactifications. In particular, it provides a new invariant for normal projective varieties isomorphic in codimension $1$. 

In the second case, we consider vector bundles with an integrable $\lambda$-connection on a normal projective variety $X$ that is (almost) liftable modulo $p^2$. In this case we restrict to varieties with $F$-liftable singularities that are analogous to log canonical singularities.  We prove that a semistable Higgs vector bundle on the regular locus of $X$, with appropriately defined vanishing Chern classes, admits a canonical Higgs--de Rham flow. This provides an analogue of some results due to D. Greb, S. Kebekus, B. Taji and T. Peternell. Finally, we give some applications of the obtained results, also to varieties defined in characteristic zero.
\end{abstract}

\section*{Introduction}

 In \cite{NS} M. S. Narasimhan and C. S. Seshadri related stable vector bundles and unitary connections on complex projective curves. Later this correspondence was generalized by S. Donaldson (see \cite{Do1} and \cite{Do2}), K. Uhlenbeck and S. T. Yau (see \cite{UY}) to  complex projective manifolds (or even compact K\"ahler manifolds) of higher dimension. Their results imply that unitary flat vector bundles are related to stable holomorphic vector bundles with vanishing rational Chern classes. The Riemann--Hilbert correspondence shows also relation to unitary representations of the topological fundamental group of the manifold. In \cite[Theorem 1.18]{DPS} the authors noticed that such (semi)stable vector bundles  with vanishing rational Chern classes correspond to numerically flat vector bundles, i.e., such that both the bundle and its dual are nef. 
 All these results can be considered as a special case of Simpson's correspondence (see \cite{Hi} for the rank $2$ case on curves and  \cite{Si1} in general) between semistable Higgs bundles with vanishing Chern classes  and representations of the topological fundamental group of a smooth, complex projective manifold.  

\medskip

The above results have various analogues for projective varieties defined over an algebraically closed field $k$ of positive characteristic $p$. In particular, strongly semistable vector bundles with vanishing (algebraic) Chern classes on smooth projective varieties correspond to numerically flat vector bundles (see \cite[Proposition 5.1]{La-S-fund}). Such bundles correspond also to representations of the so called S-fundamental group scheme introduced in \cite{BPS} in the curve case and in \cite{La-S-fund} in general.

In a different direction, an analogue of Simpson's correspondence was obtained by A. Ogus and V. Vologodsky in \cite{OV}. This correspondence works for arbitrary smooth varieties but one needs to assume existence of lifting modulo $p^2$. Moreover, semistability does not play any role in this correspondence and the correspondence depends on the choice of lifting. The notion of semistability in this context was first studied by G. Lan, M. Sheng and K. Zuo in \cite{LSZ}. In particular, they used it to establish relation between certain representations of the \'etale fundamental group  of the generic fiber of a smooth projective scheme over the Witt ring of $\bar \FF_p$ and semistable Higgs bundles with vanishing Chern classes on the special fiber. This was further generalized in \cite{SYZ} to deal with projective representations of the \'etale fundamental group. Somewhat different point of view, not restricted to bundles with vanishing Chern classes, was used in \cite{La-Inv} and \cite{La-JEMS} to establish various results like Bogomolov's inequality for semistable  Higgs sheaves,
the Miyaoka--Yau inequality and strong semipositivity theorems.

\medskip

Recently, S. Lu and B. Taji extended  the Donaldson--Uhlenbeck--Yau correspondence to complex projective varieties with klt singularities (see \cite[Theorems 1.1 and 1.4]{LT}).
Similarly, Simpson's correspondence was generalized by D. Greb et al. to smooth loci of complex projective klt varieties (see \cite{GKPT1} and \cite{GKPT3}). The main aim of this paper is to provide some analogues 
of these results (mostly of \cite{GKPT3}) in case of singular varieties (or regular loci of singular varieties), usually defined in positive characteristic. Note that even for complex manifolds some analogues of the above results fail completely. For example, numerically flat vector bundles on a big open subset of a complex projective manifold do not need to correspond to any representations of the topological fundamental group (see Example \ref{example-large-num-flat}).
Still they give rise to a new fundamental group scheme that we call a large S-fundamental group scheme.
After writing the first version of the paper, the author noticed that this group scheme was already introduced in \cite{BHS} under a different name.

To obtain appropriate analogues valid  for smooth projective varieties, we need to add
some further conditions. More precisely, we consider so called strongly numerically flat vector bundles, which  are defined as numerically flat ``with vanishing second Chern character'' (see Definition \ref{Def:strongly-num-flat} and Corollary \ref{characterization-num-flat}). This condition is natural if one notes  that the extension of a flat unitary connection from a big open subset (i.e., such that the complement has codimension $\ge 2$) of a complex manifold has a flat connection and hence vanishing rational Chern classes. In case of big open subsets of normal projective varieties, these conditions become rather delicate and we need to use Chern classes (or characters) of reflexive sheaves introduced in \cite{La-Inters}.

The following theorem generalizes, inter alia, the results of \cite{LT} to positive characteristic. In the theorem $F_X^m$ denotes the $m$-th absolute Frobenius morphism that corresponds to the $p^m$-th power mapping on $\cO_X$.

\begin{Theorem}\label{main-num-flat}
Let  $X$ be a smooth quasi-projective variety defined over an algebraically closed field of positive characteristic. Assume that $X$ admits a normal projective compactification $j:X\hookrightarrow \bar X$ with the complement $\bar X\backslash X$ of codimension $\ge 2$.
Let $\cE$ be a coherent reflexive $\cO_{X}$-module. Then the following conditions are equivalent:
	\begin{enumerate}
		\item $\cE$ is a strongly numerically flat vector bundle. 
		\item  $\cE$ is locally free and the set  $\{ (F_{X}^{m})^* \cE \} _{m\ge 0}$ is bounded.
		\item  For some ample line bundle $H$ on $\bar X$,  $j_*\cE$ is strongly slope $H$-semistable
		 and
		$$\int_{\bar X} \ch_1 (j_*\cE) H^{n-1}=\int_{\bar X} \ch_2 (j_*\cE) H^{n-2}= 0.$$
		\item  For every nef line bundle $H$ on $\bar X$,  $j_*\cE$ is strongly slope $H$-semistable
		and
		$$\int_{\bar X} \ch_1 (j_*\cE) H^{n-1}=\int_{\bar X} \ch_2 (j_*\cE) H^{n-2}= 0.$$
	\end{enumerate}	
\end{Theorem} 

Note that unlike its characteristic zero analogue \cite[Theorems 1.1 and 1.4]{LT}, 
we do not have any restrictions on the singularities of $\bar X$.
In fact, in Corollary \ref{characterization-proj-flat} we also prove a stronger theorem that deals with (strongly) projectively numerically flat vector bundles and instead of one ample line bundle we consider multipolarizations (this is important as it allows to prove vanishing of the second Chern form $\int_{\bar X} \ch_2(j_*\cE)$).  In case $X$ is a smooth complex projective variety, such a theorem for one ample polarization was proven by
N. Nakayama in \cite[Theorem 4.1]{Na}.  His proof is similar to the proof of \cite[Theorem 2]{Si}
and it depends on the Lefschetz hyperplane theorem for the topological fundamental group. 

Theorem \ref{main-num-flat} allows us to show that strongly numerically flat vector bundles form a Tannakian category with expected properties. We use this category to define the S-fundamental group scheme $\pi_1^S(X)$ that is responsible for  (extensions of irreducible) ``unitary'' representations of the (non-existent) topological fundamental group. The above theorem implies that  for a big open subset $U\subset X$ we have  $\pi_1^S(U)\simeq \pi_1^S(X)$ (but if $\bar X$ is singular, these group schemes are usually different to $ \pi_1^S(\bar X)$). In particular, if $Y$ and $Y'$ are normal projective varieties that are isomorphic in codimension $1$ then $\pi_1^S(Y_{reg})\simeq \pi_1^S(Y'_{reg})$ (see Corollary \ref{asked-by-referee}).

Theorem \ref{main-num-flat}  implies also that the pullback of a strongly numerically flat vector bundle by any morphism $f: Y\to X$, where $Y$  admits a small compactification, is also strongly numerically flat (see Theorem \ref{restriction-strong-numerical-flatness}). This fact is nontrivial even if both $X$ and $Y$ admit smooth small compactifications as it implies vanishing of Chern classes of extensions of pulled-back vector bundles.

\medskip

The main problem in proof of Theorem \ref{main-num-flat} is to show that if $\cE$ satisfies (3) then it is locally free. In case $X$ is projective,  analogous results (see  \cite[Theorem 2]{Si}, \cite[Theorem 4.1]{La-S-fund}, \cite[Theorem 11]{La-Inv}, \cite[Theorem 2.2]{La-JEMS} and \cite[Theorem A and B]{FL}) are always proven using another characterization of numerically flat vector bundles as some semistable coherent torsion free  $\cO_X$-modules with the same Hilbert polynomial as the trivial bundle of the same rank. Then one can use induction on the dimension of the variety.  This strategy fails on normal varieties because $j_*\cE$ in Theorem \ref{main-num-flat} does not need to have the same Hilbert polynomial as the trivial bundle of the same rank. Our proof still uses induction on the dimension and restriction to hypersurfaces but we encounter serious technical problems. Solution of these problems depends heavily on boundedness results from \cite{La-Bog-normal} and on a careful study of the effect of various operations on Chern classes of reflexive sheaves. 
These problems become even more difficult when proving the following theorem
(see Theorem \ref{reflexive-loc-free-Delta-Higgs} and Remark \ref{proof-ofmain-Higgs}). 
In the theorem $\int_{\bar X} \Delta (j_*V)$ denotes the discriminant form defined by the formula $$\int _{\bar X} \Delta (j_*V)=2r\int _{\bar X} c _2 (j_*V) -(r-1)\int _{\bar X} c_1 (j_*V)^2.$$
For the definition and basic properties of F-liftable singularities we refer the reader to \cite[1.2]{La-Bog-normal}
(note that the notion used in \cite{Zd} is slightly different). We also recall this notion in Section 5.

\begin{Theorem} \label{main-Higgs}
	Let  $X$ be a smooth quasi-projective variety defined over an algebraically closed field of characteristic $p$. Assume that $X$ admits a normal projective compactification $j:X\hookrightarrow \bar X$ with boundary locus $\bar X\backslash X$ of codimension $\ge 2$. Assume also that $X$ has a lifting $\tilde X$ to $W_2(k)$ and $\bar X$ has $F$-liftable singularities in codimension $2$ with respect to $\tilde X$.	
	Let  $(V, \nabla)$ be a  coherent reflexive $\cO_X$-module of rank $r\le p$ with an integrable $\lambda$-connection for some $\lambda\in k$. If for some ample line bundle $H$ on $\bar X$,  $(V, \nabla)$ is slope $H$-semistable and
	$$\int_{\bar X} \Delta (j_*V) H^{n-2}=0$$
	then $V$ is locally free, $\int_{\bar X} \Delta (j_*V)=0$ and $(V, \nabla)$ is slope semistable for every nef (multi-)polarization.
\end{Theorem}

 A \emph{\(\lambda\)-connection} on an
\(\mathcal O_X\)-module \(V\) is a \(k\)-linear morphism
\(\nabla\colon V\to V\otimes_{\mathcal O_X}\Omega_X^1\)
satisfying the \(\lambda\)-Leibniz rule
\[
\nabla(fe)
=
f\nabla (e)+\lambda e\otimes df
\]
for all local sections \(f\in\mathcal O_X\) and \(e\in V\).
It is called \emph{integrable} if its curvature
\(
\nabla^2\colon V\to V\otimes_{\mathcal O_X}\Omega_X^2
\)
vanishes. Thus a Higgs $\cO_X$-module is an $\cO_X$-module with an integrable $0$-connection and an $\cO_X$-module with an integrable connection (MIC) is an $\cO_X$-module with an integrable $1$-connection, and the above formulation applies to both Higgs and MIC cases.  For smooth projective varieties Theorem \ref{main-Higgs} is equivalent to  \cite[Theorem 2.2]{La-JEMS} (for $D=0$). The above theorem holds also in the logarithmic case (see Theorem \ref{local-freeness-log-Higgs-case}). Let us also remark that we prove a generalization of this theorem to analogues of mixed Hodge structures (see Corollary \ref{strong-log-freeness}). This generalizes \cite[Theorems 0.5 and 2.1]{La-JEMS} to the singular case.

We prove also that  if $N^{\bullet }$ is any Griffiths transverse filtration on $(V, \nabla)$ such that the associated graded $(\cE:=\Gr _{N} V, \theta )$ is  slope $H$-semistable as a Higgs sheaf then $\cE$ is locally free, $\int_{\bar X} \Delta (j_*\cE) =0$ and $(\cE, \theta )$ is slope semistable for every nef (multi-)polarization (see  Theorem \ref{reflexive-loc-free-Delta-Higgs}). This fact does not follow from the above theorem as a priori the associated graded $\cE$ does not need to be reflexive (and indeed it need not be reflexive if we do not assume vanishing of $\int_{\bar X} \Delta (j_*V) H^{n-2}$). Even when $\cE$ is reflexive, vanishing of $\int_{\bar X} \Delta (j_*\cE) H^{n-2}$ is rather difficult to show for singular $\bar X$.

This fact is interesting as every $(V, \nabla)$ as in Theorem  \ref{main-Higgs} admits a canonical filtration $S^{\bullet }$, called Simpson's filtration, satisfying the above assumptions (see Theorem \ref{existence-of-gr-ss-Griffiths-transverse-filtration}).
In particular, it is needed to prove the following corollary:

\begin{Corollary}
	In the notation of Theorem \ref{main-Higgs}, $(V, \nabla)$ admits a canonical Higgs--de Rham flow, i.e.,  
a canonical sequence of vector bundles on $X$
$${  \xymatrix{
		(V, \nabla ) \ar[rd]^{\Gr _{S}}	&& (V_0, \nabla _0)\ar[rd]^{\Gr _{F_0}}&& (V_1, \nabla _1)\ar[rd]^{\Gr _{F_1}}&\\
		&	(\cE_0, \theta_0 )=	(\Gr _{S} V, \theta)\ar[ru]^{C_{\tilde X}^{-1}}&&	(\cE_1, \theta_1)\ar[ru]^{C_{\tilde X}^{-1}}&&...\\
}}$$
with the following properties: 
\begin{enumerate}
	\item   $(V_j, \nabla_j)$ is a vector bundle with an integrable connection,
	\item  $F^{\bullet}_j$ is a  Griffiths transverse filtration on $(V_j,\nabla _j)$,
	\item  $(\cE_{j}, \theta _{j})$ is the associated graded Higgs bundle for the filtration $F^{\bullet}_{j-1}$,
	\item $(V_j, \nabla_j)$ and $(\cE_{j}, \theta _{j})$ are related by the Ogus--Vologodsky's Cartier transform.
\end{enumerate}
Moreover, all objects in the above sequence are slope semistable for every nef multi-polarization and discriminant forms of their extensions to $\bar X$ vanish.
\end{Corollary}

A Higgs--de Rham flow is a positive characteristic analogue of a Yang--Mills--Higgs flow (cf. \cite{LSYZ}). So 
the above corollary should be considered as a positive characteristic version of \cite[Theorem 1.2]{GKPT3}.
The authors of \cite{GKPT3} prove their result for varieties with klt singularities, which are in particular known to be quotient in codimension $2$. We allow somewhat weaker singularities that in characteristic zero correspond to ``log canonical in codimension $2$'' (e.g., a Zariski dense set of reductions of a cone over an elliptic curve is $F$-liftable). For a more precise study of these singularities in the surface case we refer the reader to a recent paper of T. Kawakami and T. Takamatsu \cite{KT}.

\medskip

We finish the paper with a few applications of the above results. 
They allow us to obtain some results on crystalline representations of the \'etale fundamental group of some quasi-projective varieties in the mixed characteristic. We also prove some restriction theorems for ``semistable Higgs bundles with vanishing Chern classes''
on quasi-projective varieties. In the characteristic zero case, the above results allow us to deal with varieties with more general singularities than usually considered. As in \cite{La-JEMS} our results imply also various semipositivity theorems (also for polarizable complex variations of Hodge structure on complex projective varieties with quotient singularities in codimension $2$). We leave formulation of the corresponding results to the interested reader.

\medskip

Note that in this paper we do not consider analogues of the results of \cite{GKPT1} that require several ingredients that are not yet available in positive characteristic. We plan to address this problem elsewhere.

\medskip

The structure of the paper is as follows. In Section 1 we gather a few auxiliary results. In Section 2 we recall basic results on reflexive sheaves and prove some of their properties. In Section 3 we study numerically flat vector bundles on quasi-projective varieties. In Section 4 we introduce strongly numerically flat bundles and the S-fundamental group scheme for some quasi-projective varieties. This section contains also  proof of Theorem \ref{main-num-flat}. The next section contains proof of Theorem \ref{main-Higgs}. In Section 6 we prove a logarithmic version  of this theorem. Section 7 contains some applications of these results.

\medskip

\subsection*{Notation}

Let $X$ be an integral normal  locally Noetherian scheme.
We say that an open subset $U\subset X$ is \emph{big} if every irreducible component of $X\backslash U$
has codimension $\ge 2$ in $X$.
A vector bundle on $X$ is a locally free $\cO_X$-module of finite rank. The category of vector bundles on $X$ will be denoted by $\Vect (X)$. We will also use the following definition:

\begin{Definition}
A $k$-variety $U$ is called \emph{normally big} if it is smooth and there exists an open embedding of $U$ into a normal projective $k$-variety $X$ as a big open subset. In this case $X$ is called a \emph{small compactification} of $U$.
\end{Definition}

\section{Preliminaries}

In this section we gather a few auxiliary results on reflexive sheaves (or sheaves satisfying more general Serre's conditions) and their restrictions to hypersurfaces. We also prove a certain lemma on local freeness and recall the notion of Higgs--de Rham flow.

\subsection{Reflexive sheaves}

Let $X$ be an integral normal locally Noetherian scheme. By $\Ref(\cO_X)$  we denote the category of coherent reflexive $\cO_X$-modules. It is a full subcategory of the category $\Coh(\cO_X)$ of coherent  $\cO_X$-modules. 
The inclusion functor $\Ref(\cO_X)\to \Coh (\cO_X)$ comes with a left adjoint $(\cdot )^{**}:\Coh(\cO_X)\to \Ref (\cO_X)$ given by the reflexive hull. The category $\Ref(\cO_X)$  is an additive category with kernels and cokernels and it comes with an associative and symmetric tensor product $\hat\otimes$ defined by
$$\cE\hat\otimes \cF:=(\cE\otimes _{\cO_X} \cF)^{**}.$$
We also set $\Sym ^{[m]}\cE:= (\Sym ^m \cE)^{**}$.

If $f: X\to Y$ is a proper morphism between integral, normal locally Noetherian schemes 
then we define functors $f^{[*]}:\Coh (\cO_Y)\to \Ref (\cO_X) $ and 
$f_{[*]}:\Coh (\cO_X)\to \Ref (\cO_Y) $ by composing the reflexive hull functor  $(\cdot )^{**}$ with  $f^*$ and $f_*$, respectively. In general, these functors, even after restricting to reflexive modules,  are not functorial with respect to morphisms.

If $X$ is of characteristic $p>0$ (i.e., $p\cO_X=0$) then $F_X:X\to X$ denotes the absolute Frobenius morphism. 
For $\cE\in \Ref (\cO_X)$ and a non-negative integer $m $ we set
$$F_X^{[m]} \cE:= (F_X^{m})^{[*]} \cE.$$

\subsection{Serre's conditions}

Let $X$ be an integral normal locally Noetherian scheme. Let us recall that a coherent $\cO_X$-module $\cE$ \emph{satisfies Serre's condition} $(S_m)$ if for every $x\in X$ we have
$$\mathop{\rm depth}\, \cE_x\ge \min (m, \dim \Supp \cE_{x}).$$

\begin{Proposition}\label{reflexive-standard-equivalence}
Let $\cE$ be a coherent $\cO_X$-module. Then the following conditions are equivalent.
	\begin{enumerate}
		\item $\cE$ is reflexive,
		\item $\cE$ is the dual of a coherent $\cO_X$-module,
		\item $\Supp \cE=X$ and $\cE$ satisfies condition $(S_2)$.
	\end{enumerate}
\end{Proposition}

\begin{proof}
Since the dual of a coherent $\cO_X$-module is reflexive (see \cite[Lemma 0AY4]{SP}), equivalence of (1)
and (2) is clear. Implication $(1)\Rightarrow (3)$ follows from \cite[Lemma  0AY6]{SP}. By the same lemma, to prove that $(3)\Rightarrow (1)$ it is sufficient to prove that $\cE$ is torsion free. This follows from  \cite[Lemma 0AXY]{SP}. 
\end{proof}

The following proposition generalizes \cite[Proposition 1.1.6]{HL}:

\begin{Proposition}\label{characterization-of-depth}
	Let $X$ be a separated equidimensional scheme of finite type over a field $k$
	and let $\omega_X^{\bullet}=f^{!}\cO_{\Spec k}$, where $f: X\to \Spec k$ is the structure morphism.
	Let $n=\dim X$ and let $\cE$ be a coherent $\cO_X$-module.
	Then $\cExt_X ^q(\cE, \omega_X^{\bullet}[-n])$ is a coherent $\cO_X$-module for all $q$,
	 $$\cExt_X ^q(\cE, \omega_X^{\bullet}[-n])=0$$ for $q<0$ and 
$$\codim (\Supp \cExt_X ^q(\cE, \omega_X^{\bullet}[-n]))\ge q$$ for all $q\ge 0$.
Moreover, if $\Supp \cE =X$ then the following conditions are equivalent:
	\begin{enumerate}
		\item $\cE$ satisfies condition $(S_m)$,
		\item  $\codim (\Supp \cExt_X ^q(\cE, \omega_X^{\bullet}[-n]))\ge q+m$ for all $q>0$.
	\end{enumerate}
\end{Proposition}

\begin{proof}
	By \cite[Example 0AWJ]{SP}, $\omega_X^{\bullet}$ is a dualizing complex of $X$.
	By \cite[Lemma 0AWM]{SP}  the associated dimension function is given by $\delta (x)= \dim \overline{\{x\}}=\operatorname{\rm trdeg }_k k(x)$.
	Since
	$$\cExt_X ^q(\cE, \omega_X^{\bullet}[-n])=\cExt_X ^{-(n-q)}(\cE, \omega_X^{\bullet}),$$ 
\cite[Lemma 0ECM]{SP} implies that $\cExt_X ^q(\cE, \omega_X^{\bullet}[-n])$ is  coherent and 
$$\cExt_X ^q(\cE, \omega_X^{\bullet}[-n])_x=0$$ 
for $q<0$ or $q>\dim \cO_{X,x}$. This implies the first part of the proposition.
To see equivalence of conditions (1) and (2) we  note that by \cite[Lemma 0ECM]{SP}
$$\mathop{\rm depth}\, \cE_x= \min \{i: i\ge 0  \hbox{ and } \cExt_X ^{\dim \cO_{X,x}-i}(\cE, \omega_X^{\bullet}[-n])_x\ne 0\} .$$
\end{proof}

Let us recall that $X$ is Cohen--Macaulay if and only if $\omega_X^{\bullet}[-n]=\omega_X$ is the dualizing module (see \cite[Lemmas 0AWT and 0AWU]{SP}). In general,  under the above assumptions $\omega_X^{\bullet}$ is a complex with non-zero cohomology only in degrees $-n,...,0$.

\subsection{Bertini type theorem}

Let $X$  be  an integral  normal $n$-dimensional scheme of finite type over a field $k$ and let $L$ be a line bundle on $X$. 	Let $\cE$ be a  coherent $\cO_X$-module.
We say that a section $s\in H^0(X, L)$ is \emph{$\cE$-regular} if the  map  $\cE\otimes L^{-1}\to \cE $  defined by $s$ is injective (cf. \cite[Definition 1.1.11]{HL}). 
We say that $H\in |L|$  is \emph{$\cE$-regular} if a corresponding section of $L$ is $\cE$-regular. Analogously to \cite[Corollary 1.1.14]{HL} we have the following result.

\begin{Proposition}\label{hyperplane-section-Serre's-condition}
	Let $\cE$ be a   coherent reflexive $\cO_X$-module. Let $H\in |L|$  
	and let $\imath: H\hookrightarrow X$ denote the corresponding  closed embedding.
	Then $\imath^*\cE$ is a coherent $\cO_H$-module satisfying condition $(S_1)$. Moreover, if  $H$ is $\cExt_X^q (\cE,\omega_X^{\bullet}[-n])$-regular for all $q\ge 0$ then  $\imath^*\cE$ satisfies condition $(S_2)$.
\end{Proposition}

\begin{proof}
	By \cite[Lemma 0A9Q]{SP} (``Grothendieck--Verdier duality'') we have an isomorphism
	$$\imath_*R\cHom _H(\imath^*\cE, \omega_H^{\bullet})\simeq R\cHom _X(\imath_*\imath^* \cE, \omega _X^{\bullet}).$$
	This gives an isomorphism
	$$\imath_*\cExt_H^{q} (\imath^*\cE,\omega_H^{\bullet}[-(n-1)])\simeq \cExt_X^{q+1} (\imath_*\imath^*\cE,\omega_X^{\bullet}[-n]).$$
	Composing it with the long exact sequence associated to the short exact sequence
	$$0\to\cE \otimes L^{-1} \to \cE \to \imath_*\imath^*\cE\to 0,$$
	we get
	\begin{align*}
		... &\to \cExt_X^q (\cE,\omega_X^{\bullet}[-n])\mathop{\to}^{\alpha_q} \cExt_X^q (\cE\otimes L^{-1},\omega_X^{\bullet}[-n])\mathop{\to}^{\beta_q} \imath_*\cExt_H^{q} (\imath^*\cE,\omega_H^{\bullet}[-(n-1)]) \\
		&\to \cExt_X^{q+1} (\cE,\omega_X^{\bullet}[-n])\mathop{\to}^{\alpha_{q+1}} \cExt_X^{q+1} (\cE\otimes L^{-1},\omega_X^{\bullet}[-n])\to ...\\
	\end{align*}
$H$ is equidimensional by Krull's principal ideal theorem. So Proposition \ref{characterization-of-depth} implies that $\imath^*\cE$ satisfies  condition $(S_1)$.
	If  $H$ is $\cExt_X^q (\cE,\omega_X^{\bullet}[-n])$-regular for all $q\ge 0$ then 
	$\alpha_q$ and $\alpha_{q+1}$ are injective and hence $\beta_q$ is surjective. Therefore we have
	$$\imath_*\cExt_H^{q} (\imath^*\cE,\omega_H^{\bullet}[-(n-1)]) \simeq \imath_*\imath^*\cExt_X^q (\cE,\omega_X^{\bullet}[-n])\otimes _{\cO_X} L$$ 
and since $\cE$ satisfies condition $(S_2)$ (see Proposition \ref{reflexive-standard-equivalence}), Proposition \ref{characterization-of-depth} implies that
	$$\cExt_H^{q} (\imath^*\cE,\omega_H^{\bullet}[-(n-1)]) \simeq \imath^*(\cExt_X^q (\cE,\omega_X^{\bullet}[-n])\otimes _{\cO_X} L)$$
has codimension $\ge (q+2)$ in $H$. By the same proposition $\imath^*\cE$ satisfies condition $(S_2)$.
\end{proof}

The above proposition, \cite[Lemma 0AXY]{SP} and Proposition \ref{reflexive-standard-equivalence} 
imply the following corollary:

\begin{Corollary}\label{hyperplane-section-of-reflexive}
	Let $\cE$ be a   coherent reflexive $\cO_X$-module. Let $H\in |L|$ be irreducible and normal, 
	and let $\imath: H\hookrightarrow X$ denote the corresponding  closed embedding.
	Then $\imath^*\cE$ is torsion free. Moreover, if  $H$ is $\cExt_X^q (\cE,\omega_X^{\bullet}[-n])$-regular for all $q\ge 0$ then  $\imath^*\cE$ is reflexive.
\end{Corollary}
\medskip

We will need the following version of Bertini's theorem with basepoints.

\begin{Theorem}\label{Bertini's-theorem}
	Let $X$ be a normal projective variety of dimension $n\ge 2$ defined over an algebraically 
	closed field $k$. Let $L$ be a very ample line bundle and let $\cE$ be a coherent reflexive $\cO_X$-module. Let us also fix a $0$-dimensional subset $Z\subset X_{reg}$ (possibly empty).
	Then for large integers  $m$, a general hypersurface $H\in  |L^{\otimes m}|$ containing $Z$ satisfies the following conditions:
	\begin{enumerate}
		\item $H$ is irreducible and normal,
		\item $H_{reg}= X_{reg}\cap H$,
		\item $\cE|_H$ is torsion free on $H$,
		\item $\cE|_H$ is reflexive on $H\backslash Z$.
	\end{enumerate}
\end{Theorem} 

\begin{proof}
Let $f: \tilde X\to X$ be the blow up of $X$ along $Z$ (taken with the reduced subscheme structure) and let $E$ denote the exceptional divisor of $f$. There exists $m_0$ such that for all $m\ge m_0$, the line bundle $f^*L^{\otimes m}(-E)$ is very ample
(see, e.g., \cite[Chapter II, Exercise 7.14]{Ha}). Let $H'\in |f^*L^{\otimes m}(-E)|$ be a general divisor.
Then $H=f(H')$ is a general divisor in $ |L^{\otimes m}|$ containing $Z$. 
By the usual Bertini's theorem $H'$ intersects $\tilde X_{reg}$ transversally, so 
$H_{reg}$ contains $(X_{reg}\cap H)\backslash Z$.  Moreover, by Seidenberg's version of Bertini's 
theorem $H'$ is irreducible and normal.
Since for any point $x\in Z$, $H'$ intersects  $f^{-1}(x)$ transversally, the divisor $H$ is smooth at $x$ (we use the fact that $\cO_{f^{-1}(x)}(H')\simeq \cO_{\PP^{n-1}}(1)$). Therefore 
$H_{reg}\supset X_{reg}\cap H$ and  $H$ is irreducible and normal, so we have (1). Since $H$ is locally principal, it cannot be smooth at singular points of $X$ and hence $H_{reg}= X_{reg}\cap H$ proving (2). (3) follows from Corollary \ref{hyperplane-section-of-reflexive}.
For any fixed coherent $\cO_{\tilde X}$-module $\cF$, $H'$ is $\cF$-regular (see 
\cite[Lemma 1.1.12]{HL}). So Corollary \ref{hyperplane-section-of-reflexive} implies also  that 
the restriction of $f^{[*]}\cE$ to $H'$ is reflexive. This implies (4).
\end{proof}

\medskip

\begin{Remark}
We will also use the fact that if $Y\subset X$ is a closed subset  with $\dim Y>0$ (or $Y$ is $0$-dimensional and $Z\cap Y=\emptyset$)
then 
for a general hypersurface $H\in  |L^{\otimes m}|$ containing $Z$ we have $\dim (Y\cap H)<\dim Y$ ($Y\cap H=\emptyset$, respectively). This follows immediately from the proof of the above theorem.
\end{Remark}

\medskip

\begin{Example}
	Let $\cE$ be a coherent reflexive $\cO_X$-module. In general, it is not true that one can  find a normal  hypersurface $H\in  |L^{\otimes m}|$ containing some fixed closed point $x\in X$ such that 
	the restriction $\cE|_H$ is reflexive as an $\cO_H$-module.
	For example, on $X=\PP^3$ we consider the sequence
	$$0\to \cE\to \cO_{\PP^3}^{\oplus 3}(-1)\mathop{\to}^{\alpha} m_x\to 0,$$
	where $x=[0,0,0,1]$ and the map $\alpha$ is given by $(x_0,x_1,x_2)$. Then $\cE$ is a rank $2$ reflexive $\cO_{X}$-module but the restriction of $\cE$ to any irreducible normal effective Cartier divisor $H\subset X$ containing $x$ is not reflexive. Indeed, $\cE|_H$ is torsion free so the sequence
	$$0\to \cE|_H\to \cO_{H}^{\oplus 3}(-1)\mathop{\to}^{\alpha|_H} m_x|_H\to 0$$
is exact. However, $m_x|_H$  contains a non-zero $0$-dimensional subsheaf $\Tor _1^X(\cO_x, \cO_H)\simeq \cO_x$. So $\cE|_H$ is not saturated in $\cO_{H}^{\oplus 3}(-1)$ and hence it is not reflexive as an $\cO_H$-module.	
\end{Example}

\subsection{Criterion for local freeness}

\begin{Lemma}\label{local-duality}
	Let $R$ be a regular local Noetherian ring of dimension $n$ and let $M$ be a finitely generated $R$-module.
	Then $$\Ext _R^q(M,R)=0$$	 
	for $q<\codim M=n-\dim M$.
\end{Lemma}

\begin{proof}
	Let $m$ be the maximal ideal in $R$.
	By Grothendieck's vanishing theorem $H^j_m(M)=0$ for $j>\dim M$.
	So the required vanishing follows from the local  duality theorem.
\end{proof}

The following lemma generalizes \cite[Lemma 1.12]{La-JEMS}.

\begin{Lemma}\label{loc-free-passing-to-graded}	
	Let $X$ be a smooth variety  of dimension $n\ge 3$
	defined over an arbitrary algebraically closed field $k$.  
	Let  $\cE$ be a coherent reflexive $\cO_X$-module with a filtration  $\cE=N^0\supset
	N^1\supset ...\supset N^s=0$ such that the associated graded $\Gr
	_N(\cE)=\bigoplus _i N^i/N^{i+1}$ is torsion free. Let $\cF$ be the reflexive hull of  $\Gr
	_N(\cE)$ and let us assume that 
	\begin{enumerate}
		\item  $\cF$ is locally free, and
		\item $\Gr _N(\cE)$ is locally free outside a closed subset of codimension $\ge 3$. 
	\end{enumerate}
	Then both $\cE$ and $\Gr _N(\cE)$ are locally free and the canonical map $\Gr _N(\cE)\to \cF$ is an isomorphism.
\end{Lemma}

\begin{proof}
	As in \cite{La-JEMS} it is sufficient to prove the lemma assuming that $s=2$.
	Then $N^1$ is locally free and we have a short exact sequence
	$$0\to N^0/N^1\to (N^0/N^1)^{**}\to T\to 0$$
	for some sheaf $T$ supported in codimension $\ge 3$. By
	assumption we also know that $(N^0/N^1)^{**}$ is locally free.  
	By Lemma \ref{local-duality} we have $\cExt ^2_X(T, N^1)=0$. 
	Hence by the long $\cExt$ exact sequence, the canonical	map 
	$$\cExt ^1_X((N^0/N^1)^{**}, N^1)\to \cExt ^1_X(N^0/N^1, N^1)$$ 
	is surjective. But $(N^0/N^1)^{**}$ is locally free, so $\cExt ^1_X((N^0/N^1)^{**}, N^1)=0$.
	Therefore $ \cExt ^1_X(N^0/N^1, N^1)=0$ and the sequence
	$$0\to N^1\to N^0\to N^0/N^1\to 0$$
	is locally split. Since $N^0$ is reflexive, $N^0/N^1$ is reflexive. This together with our assumptions gives the required assertion.
\end{proof}

\subsection{Higgs--de Rham flows}

Let $X$ be a smooth variety defined over an algebraically closed field $k$ of characteristic $p>0$. 
Assume that $X$ has a lifting $\tilde X$ to $W_2(k)$. In this case Ogus and Vologodsky in \cite{OV} 
defined the Cartier transform $C_{\tilde X}$, which is an equivalence between the category
of vector bundles with an integrable connection with nilpotent $p$-curvature of level $<p$ and the category 
of nilpotent Higgs bundles of level $<p$. Its quasi-inverse is denoted by  $C^{-1}_{\tilde X}$ and called the inverse Cartier transform.

\medskip

In this paper we use the following definition, modifying somewhat the one used in \cite{LSZ}, \cite{SYZ} etc.

\begin{Definition}\label{def:Higgs-de_Rham-flow}
	Let $(V, \nabla)$ be a vector bundle  with an integrable $\lambda$-connection for some $\lambda \in k$ (see Introduction or Subsection 1.6).
	Assume that $V$ has rank $\le p$ and let us set $(V_{-1}, \nabla_{-1}):=(V, \nabla)$.
	A \emph{Higgs--de Rham flow} of $(V, \nabla)$  is a sequence of the form
	$${  \xymatrix{
			(V_{-1}, \nabla _{-1}) \ar[rd]^{\Gr _{F_{-1}}}	&& (V_0, \nabla _0)\ar[rd]^{\Gr _{F_0}}&& (V_1, \nabla _1)\ar[rd]^{\Gr _{F_1}}&\\
			&	(\cE_0, \theta _0)\ar[ru]^{C_{\tilde X}^{-1}}&&	(\cE_1, \theta_1)\ar[ru]^{C_{\tilde X}^{-1}}&&...\\
	}}$$
	where 
	\begin{enumerate}
		\item  for $j\ge 0$, $(V_j, \nabla_j)$ is a vector bundle endowed with an integrable connection,
		\item for $j\ge -1$, $F^{\bullet}_j$ is an exhaustive decreasing filtration indexed by $\ZZ$ on $V_j$ by vector subbundles, satisfying Griffiths transversality
		$$\nabla_j F_j^l\subset F_j^{l-1}\otimes \Omega_X,$$
		\item  $(\cE_{j}, \theta _{j})$ is the associated graded Higgs bundle for the filtration $F^{\bullet}_{j-1}$,
		\item $(V_j, \nabla_j)=C^{-1}_{\tilde X}(\cE_{j}, \theta _{j})$ for all $j\ge 0$.
	\end{enumerate}
\end{Definition}

\medskip

An analogous definition can be also made in the logarithmic case. Note that we allow a Higgs--de Rham flow to begin with any vector bundle with an integrable $\lambda$-connection, whereas a Higgs--de Rham flow in  \cite{LSZ} and \cite{SYZ} starts with the term $(\cE_0, \theta_0)$, which is a system of Hodge bundles (``a graded Higgs bundle'').

In this paper we also use a similar notion of a \emph{Higgs--de Rham sequence}, but in this case we do not require $V_j$ or $E_j$ to be vector bundles and we consider them on possibly non-smooth but projective 
varieties (see Sections 5 and 6).

\subsection{Simpson's filtration}

Let $X$ be a normal projective variety defined over an algebraically closed field $k$ and let $D$ be an effective reduced Weil divisor on $X$. Then we define $ T_X (\log D)$ as the reflexive extension of the corresponding logarithmic tangent bundle on the locus where $(X,D)$ is a smooth variety with a simple normal crossing divisor. Clearly, it is a subsheaf of the reflexive sheaf $T_{X/k}$.

In the following we use standard notions of Lie algebroids, modules over Lie algebroids and semistability as explained in \cite{La-Lie-alg}.
Let $[ \cdot, \cdot]_{T_X/k}$ be the standard Lie bracket on $T_{X/k}$ and let  $\alpha: T_X (\log D)\hookrightarrow T_{X/k}$ denote the canonical embedding. Let us fix $\lambda\in k$ and consider a Lie algebroid $T_{\lambda}(X,D)$, with the underlying $\cO_X$-module $T_X(\log D)$, the Lie bracket  $\lambda \, [ \cdot, \cdot]_{T_X/k}$ and the anchor map $\lambda \alpha$.

\emph{An $\cO_X$-module with an integrable logarithmic $\lambda $-connection} on $(X,D)$
is a  $T_{\lambda}(X,D)$-module. If $\lambda=0$ then an $\cO_X$-module with an integrable logarithmic $\lambda $-connection is called \emph{a logarithmic Higgs sheaf} on $(X,D)$. For $\lambda=1$ we talk about  $\cO_X$-modules with an integrable logarithmic connection. In case $D=0$, we omit the word ``logarithmic''.
Semistability of $\cO_X$-modules with an integrable logarithmic $\lambda $-connection is defined as usual
with respect to  $T_{\lambda}(X,D)$-submodules. 

Let us recall the following special case of \cite[Theorem 5.6]{La-Bog-normal}:

\begin{Theorem} \label{existence-of-gr-ss-Griffiths-transverse-filtration}
	Let $(V, \nabla)$ be a  slope semistable torsion free coherent $\cO_X$-module with an integrable logarithmic $\lambda $-connection. Then there exists a canonically defined slope gr-semistable Griffiths transverse
	filtration $S^{\bullet}$ on $(V, \nabla)$ such that the associated graded $(\Gr _{S} V, \theta )$ is  a slope semistable torsion free Higgs sheaf.
\end{Theorem}

The above filtration $S^{\bullet}$ is called \emph{Simpson's filtration}. This filtration is usually non-trivial even if $(V, \nabla)$ is a Higgs $\cO_X$-module (since the associated graded is always a system of Hodge sheaves).  Note that if $V$ is locally free on a smooth projective variety, the associated graded  does not need to be reflexive (and hence it is not locally free).  So even for smooth varieties the above filtrations do not immediately give rise to a Higgs-de Rham flow.  If $V$ is locally free with vanishing rational Chern classes on a smooth complex projective variety then $\Gr _{S} V$ is locally free (but this is a non-trivial fact equivalent to \cite[Theorem 2]{Si}). In Sections 5 and 6 we generalize this fact and use the above filtrations to construct Higgs--de Rham flows on the regular part of $X$ for reflexive sheaves with ``vanishing discriminant''.

\section{Reflexive sheaves on normal varieties}

In this section we recall the basic properties of Chern classes of reflexive sheaves and prove some auxiliary results on filtered sheaves. We also prove a certain lemma on bounded families of sheaves on normally big varieties. 

\subsection{Chern classes of reflexive sheaves}

Let $X$ be a normal projective variety of dimension $n$ defined over an algebraically closed field $k$ of characteristic $p>0$. 
Let $N^1(X)$ be the group of line bundles modulo numerical equivalence. This coincides with the group of linear (or algebraic) equivalence classes of Cartier divisors modulo $\Pic^{\tau}(X)$. In the following, we also write $\bar B^1(X)$ for the group of algebraic equivalence classes of Weil divisors modulo torsion.
If $\cE$ is a coherent torsion free $\cO_X$-module of rank $r$ then we write $c_1(\cE)$ for the divisor class of $\det \cE :=(\bigwedge ^r \cE )^{**}$.

In \cite{La-Inters} the author introduced  intersection product of two Weil divisors 
$D_1$ and $D_2$ with $(n-2)$ line bundles and intersection products  with the second Chern character of a coherent reflexive $\cO_X$-module $\cE$. In particular,  we have  $\ZZ$-multilinear symmetric 
forms
$$D_1.D_2.: N^1(X)^{n-2}\to \QQ, \quad (L_1,...,L_{n-2})\to D_1.D_2.L_1...L_{n-2}$$
and
$$\int_X \ch_2(\cE): N^1(X)^{n-2}\to \RR, \quad (L_1,...,L_{n-2})\to \int _X \ch _2 (\cE)L_1...L_{n-2}.$$
These intersection products give the same numbers under the base field extension and their computation in higher dimensions can be reduced to the surface case. In the surface case, intersection product coincides with  Mumford's intersection product.  The second Chern character on surfaces satisfies the following Riemann--Roch inequality.  There exists some constant $C$ such that for every rank $r$  coherent reflexive $\cO_X$-module $\cE$ on $X$ we have
$$\left|\chi (X, \cE) -\left(\int_X\ch_2 (\cE)-\frac {1}{2}c_1 (\cE). K_X +r\chi (X, \cO_X)\right)\right|\le C r^2.$$
We  can also consider the forms $\int_X  c _1 (\cE)^2, \int_Xc_2(\cE), \int _X \Delta (\cE): N^1(X)^{n-2}\to \RR$
defined by sending $(L_1,...,L_{n-2})$ to
$$ \int _X c _1 (\cE)^2L_1...L_{n-2}:= c_1(\cE).c_1(\cE).L_1...L_{n-2},$$
$$\int _X c _2 (\cE)L_1...L_{n-2}:=\frac{1}{2}\int _X c _1 (\cE)^2L_1...L_{n-2} -\int _X \ch _2 (\cE)L_1...L_{n-2}$$
and 
$$\int _X \Delta (\cE)L_1...L_{n-2}:=2r\int _X c _2 (\cE)L_1...L_{n-2} -(r-1)\int _X c_1 (\cE)^2L_1...L_{n-2}.$$

\medskip

The above properties imply the following lemma (see \cite[1.4]{La-Inters} for the notation):

\begin{Lemma}\label{computation-Delta}
Let $\cE$ be a coherent reflexive $\cO_X$-module and let us write $p^m=q_mr+r_m$, where $q_m$ and $r_m$ are integers such that $0\le r_m<r$. 
	Then for any line bundles $L_1,...,L_{n-2}$ on $X$ we have
	$$\int _X \Delta (\cE)L_1...L_{n-2}= -2r\lim _{m\to \infty }\frac {  \chi (X,  c_1(L_1)...c_1(L_{n-2})\cdot F_X^{[m]}\cE (-q_mc_1(\cE )) ) }{p^{2m}}.$$
\end{Lemma}
 
\begin{proof}
First assume that $X$ is a surface. Then there exists a constant $C$ depending only on $X$ such that for every rank $r$  coherent reflexive $\cO_X$-module $\cE$ on $X$ we have
	$$\left|\chi (X, \cE) -\left(\frac {1}{2r}c_1 (\cE). (c_1(\cE) -rK_X)- \frac {1}{2r}\int_X\Delta  (\cE) +r\chi (X, \cO_X)\right)\right|\le C r^2.$$
	Since
	$$\int_X\Delta  ( F_X^{[m]}\cE (-q_mc_1(\cE )))=\int_X\Delta  ( F_X^{[m]}\cE)= p^{2m}\int_X\Delta  (\cE),$$
	we get
	$$\left|\chi (X, F_X^{[m]}\cE (-q_mc_1(\cE )) )-\left(\frac {1}{2r}r_mc_1 (\cE). (r_mc_1(\cE) -rK_X)- \frac {p^{2m}}{2r}\int_X\Delta  (\cE) +r\chi (X, \cO_X)\right)\right|\le C r^2.$$
	Dividing by $p^{2m}$ and passing with $m$ to infinity gives the required equality on $X$.

Now let us consider the general case. By linearity of both sides of the equality, we can assume that all $L_i$ are very ample. Then we can follow the above proof using \cite[Theorem 5.5]{La-Inters} instead of the above Riemann--Roch inequality on surfaces.
\end{proof}

Let us fix a collection $(L_1,...,L_{n-2})$ of line bundles on $X$. We will need the following results.

\begin{Lemma}\label{inequality-on-Delta} 
	Assume that $n\ge 3$ and $L_2,...,L_{n-2}$ are ample.
	Let $H\in |L_1|$ be a normal hypersurface and let $\imath: H\hookrightarrow X$ be the closed embedding.
	Let us set  $\cF:=(\imath ^*\cE)^{**}$.  Then 
	$$\int_X \Delta (\cE) L_1...L_{n-2}\ge \int_H \Delta (\cF) \imath^*L_2...\imath^*L_{n-2}$$
	and  if equality holds then $(\cF/\imath ^*\cE)|_{H_{reg}}$ is supported in codimension $\ge 3$ in $H_{reg}$. 
\end{Lemma}

\begin{proof}
	We can assume that $L_2,...,L_{n-2}$ are very ample. Taking a base field extension $k\subset K$ to some uncountable algebraically closed field $K$,
	we can also assume that $k$ is uncountable (see \cite[Theorem 0.4]{La-Inters}). Let $Y\in |L_2|\cap ...\cap |L_{n-2}|$ be a very general $3$-fold. Let $\Cjmath : Y\hookrightarrow X$ and $\Cjmath _H: H\cap Y\hookrightarrow H$
	denote the closed embeddings.
	Then $\Cjmath^*\cE$ and $\Cjmath_H^*\cF$ are reflexive by Corollary \ref{hyperplane-section-of-reflexive}, and
	$$\int_X  \Delta  (\cE) L_1...L_{n-2}= \int_Y  \Delta  (\Cjmath^*\cE)\Cjmath ^*L_1$$	
	and 
	$$ \int_H  \Delta (\cF) \imath^*L_2...\imath^*L_{n-2}= \int_{H\cap Y} \Delta  (\Cjmath_H^*\cF)$$
by \cite[Theorem 0.4]{La-Inters}.	
	So in the following we can assume that $X$ is a $3$-fold. For all $m\ge 0$
let us set
	$\cE_m:= F_X^{[m]}\cE (-q_mc_1(\cE ))$ and $\cF_m:= F_H^{[m]}\cF (-q_mc_1(\cF))$, where $q_m$ is as in Lemma \ref{computation-Delta}.
	We have short exact sequences
	\begin{align*}
		0\to \cE_m \otimes L_1^{-1}\to \cE _m\to\imath _* \imath ^*(\cE_m)\to 0.
	\end{align*} 
By Corollary \ref{hyperplane-section-of-reflexive}, $\imath ^*\cE_m$ is a coherent torsion free $\cO_{H}$-module 
and it is isomorphic to $\cF_m$ outside a finite number of points where either $H$ is singular or $\imath ^*\cE_m$ is not locally free. So we have  $(\imath ^*\cE_m)^{**}\simeq \cF_m$. Let $\cG_m$ be the cokernel of the canonical map  $\imath ^*\cE_m\to\cF_m$.	Since $H$ is a surface,	$\cG_m$ is supported on a finite set of points.
	Since  $F_{H_{reg}}$ and $F_{X_{reg}}$ are flat, we see that	
	$$\cE_m|_{X_{reg}}= (F_X^{m})^*\cE (-q_mc_1(\cE ))|_{X_{reg}}$$ and 
	$$\cF_m |_{H_{reg}}= (F_H^{m})^*\cF (-q_mc_1(\cF)) |_{H_{reg}}.$$
	It follows that 
	$$\cG_m |_{H_{reg}}=(F_H^m)^*\cG_0 |_{H_{reg}}$$
and hence
	$$\chi(X, \imath_* \cG_m)= h^0 (X, \imath_* \cG_m)\ge h^0 (X,  \imath_* ((F_H^m)^*\cG _0|_{H_{reg}}))=
	p^{2m} h^0 (X,  \imath_* (\cG_0|_{H_{reg}})).$$
	Therefore we have
	$$	\chi(X,	c_1(L_1)\cdot \cE_m ) = \chi (H, \imath ^* \cE_m)
	=\chi (H, \cF _m)-\chi(X, \imath_* \cG_m)	\le 
	\chi (H, \cF_m )-p^{2m} h^0 (X,  \cG_0|_{H_{reg}}).$$
But by  Lemma \ref{computation-Delta} we have
	\begin{align*}
		\lim _{m\to \infty }\frac{ 	\chi(X,	c_1(L_1)\cdot \cE_m ) }{p^{2m}}
	=-\frac{1}{2r}\int_X \Delta (\cE)L_1,
	\end{align*}
so dividing the above inequality by $p^{2m}$ and taking the limit as $m\to \infty$ gives
	\begin{align*}
	-	\frac{1}{2r}\int_X \Delta (\cE)L_1+h^0 (X,  \cG_0|_{H_{reg}})\le	\lim _{m\to \infty }\frac{  \chi (H, \cF_m )}{p^{2m}}= -\frac{1}{2r}\int_H \Delta (\cF).
	\end{align*}
If 	$\int_X \Delta (\cE)L_1= \int_H \Delta (\cF)$ then $h^0 (X,  \cG_0|_{H_{reg}})=0$, so $\cG_0|_{H_{reg}}=0$. This proves the required assertions.
\end{proof}

\medskip

\begin{Lemma}\label{Chern-classes-filtrations}
	Assume that all $L_1,...,L_{n-2}$ are ample.
	Let $\cE$ be a reflexive coherent $\cO_X$-module and let $\cE=N^0\supset
N^1\supset ...\supset N^s=0$ be a filtration such that the associated graded $\Gr
_N(\cE)=\bigoplus _i N^i/N^{i+1}$ is torsion free. Let $\cF$ be the reflexive hull of  $\Gr
_N(\cE)$. If  $$\int _X \Delta  (\cE)L_1...L_{n-2}= \int _X \Delta (\cF)L_1...L_{n-2} $$
then the intersection of the support of the cokernel of $\Gr_N(\cE)\to \cF$ with $X_{reg}$ has   codimension $\ge 3$ in $X_{reg}$. Moreover, if $\cF|_{X_{reg}}$ is   locally free then the canonical map $\Gr
_N(\cE)|_{X_{reg}}\to \cF|_{X_{reg}}$ is an isomorphism and $\cE|_{X_{reg}}$ is locally free. 
\end{Lemma}

\begin{proof}
	Since $c_1(\cE)=c_1(\cF)$, our assumption is equivalent to 
	$$\int _X \ch_2  (\cE)L_1...L_{n-2}= \int _X \ch_2 (\cF)L_1...L_{n-2} .$$
	By induction we can reduce the assertion to the case when $s=2$. 
		Let $\cE, \cE_1, \cE_2$ be coherent reflexive $\cO_X$-modules and let
	$$0\to \cE_1\to \cE \to \cE_2\to \cG \to 0$$
	be an exact sequence of coherent $\cO_X$-modules with $\cG$
	 supported on a closed subscheme $Z$ of codimension $\ge 2$ in $X$.
	We need to show that if
	$$\int _X \ch _2 (\cE)L_1...L_{n-2} = \int _X \ch _2 (\cE _1)L_1...L_{n-2} +\int _X \ch _2 (\cE _2)L_1...L_{n-2} $$
	then $Z\cap X_{reg}$ has codimension $\ge 3$ in $X_{reg}$.
	
	We have exact sequences  
	$$0\to F_X^{[m]}\cE_1\to F_X^{[m]} \cE \to F_X^{[m]} \cE_2 \to \cG_m\to 0 ,$$
	where $\cG_m$ is supported on a closed subset of $X$ of codimension $\ge 2$.
	Let  $j: U:=X_{reg}\hookrightarrow X$ denote the open embedding. Since $F_U^*$ is exact, the sequences 
	$$0\to (F_U^m)^*\cE_1\to (F_U^m)^* \cE \to (F_U^m)^*\cE_2\to (F_U^m)^*(\cG|_U) \to 0$$
	are exact. This shows that $\cG_m|_U=  (F_U^m)^*(\cG|_U)$. 
	
	Without loss of generality we can also assume that all $L_i$ are very ample. 
	Let us restrict to a very general complete intersection surface $S\in |L_1|\cap ...\cap |L_{n-2}|$.
	By Bertini's theorem it is sufficient to prove that $Z\cap S_{reg}=\emptyset$. 
	
	Note that $\cG_m |_S$ is supported on a finite set of points, so
	$$\chi (S, \cG_m |_S)\ge \chi (U\cap S, (F_U^m)^*\cG |_{U\cap S})=p^{2m} h^0(S\cap U, \cG |_{U\cap S}).$$
	
	So we get inequalities 
	$$ \chi (S, F_X^{[m]} \cE |_S) \le  \chi (S, F_X^{[m]} \cE_1 |_S)+ \chi (S, F_X^{[m]} \cE _2|_S) -p^{2m} h^0(S\cap U, \cG |_{U\cap S}).$$
	Dividing by $p^{2m}$, passing to the limit and using \cite[Theorem 5.2]{La-Inters}, we get 
	$$\int _X \ch _2 (\cE)L_1...L_{n-2} \le  \int _X \ch _2 (\cE _1)L_1...L_{n-2} +\int _X \ch _2 (\cE _2)L_1...L_{n-2} - h^0(S\cap U, \cG |_{U\cap S}).$$
	So our assumption implies that $ h^0(S\cap U, \cG |_{U\cap S})=0$. Therefore $Z\cap X_{reg}\cap S=\emptyset$, which implies the first assertion.
	The last part follows from the first one  and from  Lemma \ref{loc-free-passing-to-graded}.
\end{proof}

\subsection{Restriction of filtrations to hypersurfaces}

The following technical lemma plays an important role in the proof of Theorem \ref{reflexive-loc-free-Delta-Higgs}.

\begin{Lemma}\label{passing-to-grading-for-hypersurfaces}	
	Let $X$ be a smooth variety  of dimension $n\ge 3$
	defined over an arbitrary algebraically closed field $k$.  
	Let $\cE$ be a reflexive coherent $\cO_X$-module and let $\cE=N^0\supset
	N^1\supset ...\supset N^s=0$ be a filtration such that the associated graded $\Gr
	_N\cE=\bigoplus _i N^i/N^{i+1}$ is torsion free. Let $\imath: H\hookrightarrow X$ be a smooth effective Cartier divisor. Let us set $\cF:=(\Gr _N\cE)^{**}$, $\tilde \cE:= (\imath ^*\cE)^{**}$
	and  $\tilde \cF:= (\imath ^*\cF)^{**}$.
	Let  $\tilde N^{\bullet}$ be the filtration of $\tilde \cE$ defined by taking as $\tilde N^i$
	the saturation of $\imath ^*N^i\subset \imath ^*\cE\subset \tilde \cE$. 
	Assume that
	\begin{enumerate}
		\item both $\imath ^* \cE$ and $\imath ^* \cF$ are reflexive outside a closed point $x\in H$,
		\item $\imath ^* (\Gr _N\cE)$ is torsion free outside $x$,
		\item $\Gr _{\tilde N}\tilde E$ is reflexive,
		\item $\dim \Supp (\cF/\Gr _N(\cE))=0$.
	\end{enumerate}
	Then 
	$h^0( \tilde \cE/\imath ^* \cE)= h^0(\tilde \cF/\imath ^* \cF  ).$
\end{Lemma}

\begin{proof}
	We have a commutative diagram
	$$
	\xymatrix{
		&0&0&0&\\
		&\tilde N^{i+1}/\imath ^*N^{i+1}\ar[r]\ar[u]&\tilde N^i/\imath ^*N^{i}\ar[r]\ar[u] &\coker\alpha_i\ar[r]\ar[u]&0\\
		0\ar[r]&\tilde N^{i+1}\ar[r]\ar[u]&\tilde N^i\ar[r]\ar[u] &\tilde N^i/\tilde N^{i+1}\ar[r]\ar[u]&0\\
		0\ar[r]&\imath ^*N^{i+1}\ar[r]\ar[u]&\imath ^*N^{i}\ar[r]\ar[u] &\imath ^*(N^{i}/ N^{i+1})\ar[r]\ar[u]^{\alpha  _i}&0\\
		0\ar[r]&0\ar[u]\ar[r]&0\ar[u]\ar[r]&\ker\alpha _i\ar[u]&\\
		&0\ar[u]&0\ar[u]&0\ar[u]&\\
	}
	$$
	Let us set $\alpha:=\bigoplus \alpha_i:  \imath ^* (\Gr _N\cE)\to \Gr_{\tilde N}\tilde \cE$.
	Note that outside $\Supp (\cF/\Gr _N(\cE))\cup \{x\}$, which by assumption is a finite set of points, the sheaf $\imath ^* (\Gr _N\cE)$ is reflexive. So kernels and cokernels of all $\alpha_i$ are supported in dimension $0$ and by the snake lemma we have
	\begin{align*}
		h^0(\tilde\cE/\imath ^* \cE)&= h^0(\tilde N^0 /\imath ^*N ^0)=h^0( \tilde N^1/\imath ^*N ^1) +
		h^0(\coker \alpha _0)- h^0(\ker \alpha  _0)\\
		...&= h^0(\coker \alpha )- h^0(\ker \alpha).
	\end{align*}	
	Let  $\beta $ denote the canonical map $\imath ^* (\Gr _N\cE)\to \imath ^* \cF$ and let 
	$\gamma: \imath ^* (\Gr _N\cE) \to \tilde \cF $ be the composition of $\beta$ with the reflexivization
	$ \imath ^* \cF\to \tilde \cF$. 
	Note that outside a finite set of points both $\alpha$ and $\gamma$ are isomorphisms. Since both $\Gr_{\tilde N}\tilde \cE$ and $\tilde \cF$ are reflexive and isomorphic in codimension $1$, 
	we have an isomorphism $ \Gr_{\tilde N}\tilde \cE\to \tilde \cF$ and 
	\begin{align*}
		h^0(\tilde\cE/\imath ^* \cE)= h^0(\coker\gamma )- h^0(\ker \gamma).
	\end{align*}
	Since $\imath ^* \cF$ is torsion free (by Corollary \ref{hyperplane-section-of-reflexive}), the map $\imath ^* \cF\to \tilde \cF$ is injective and hence
$\ker \gamma=\ker \beta.$ So we have a short exact sequence
	$$0\to  \imath ^* (\Gr _N\cE)/\ker\gamma \to \imath ^* \cF\to \imath ^* (\cF/\Gr _N\cE  )\to 0 .$$
	Now consider the commutative diagram 
	$$
	\xymatrix{
		&0&0&0&\\
		&\imath ^* (\cF/\Gr _N\cE  )\ar[r]\ar[u]&0\ar[r]\ar[u] &0\ar[r]\ar[u]&0\\
		0\ar[r]&\imath ^* \cF\ar[r]\ar[u]&\tilde \cF\ar[r]\ar[u] &\tilde \cF/\imath ^* \cF\ar[r]\ar[u]&0\\
		0\ar[r]&\imath ^* (\Gr _N\cE)/\ker \gamma\ar[r]\ar[u]&\tilde \cF\ar[r]\ar[u] &\coker \gamma \ar[r]\ar[u]^{\delta }&0\\
		&0\ar[u]\ar[r]&0\ar[u]\ar[r]&\ker\delta \ar[u]&\\
		&0\ar[u]&0\ar[u]&0\ar[u]&\\
	}
	$$
	By the snake lemma $\ker \delta \simeq \imath ^* (\cF/\Gr _N(\cE)  )$, so we have 
	\begin{align*}
		h^0(\tilde \cF/\imath ^* \cF  )= h^0(\coker \gamma)- h^0(\imath ^* (\cF/\Gr _N(\cE))).
	\end{align*}
	So to prove the lemma it is sufficient to prove that  $h^0(\ker \beta )=h^0(\imath ^* (\cF/\Gr _N(\cE)))$. To do so let us consider the commutative diagram 
	$$
	\xymatrix{
		&0&0&0&\\
		&(\cF/\Gr _N(\cE)) (-H)\ar[r]\ar[u]&\cF/\Gr _N(\cE) \ar[r]\ar[u] &\imath_*\imath ^*(\cF/\Gr _N(\cE))\ar[r]\ar[u]&0\\
		0\ar[r]&\cF(-H)\ar[r]\ar[u]&\cF\ar[r]\ar[u] &\imath_*\imath ^* \cF \ar[r]\ar[u]&0\\
		0\ar[r]&\Gr _N(\cE) (-H)\ar[r]\ar[u]&\Gr _N(\cE)\ar[r]\ar[u] &\imath_*\imath ^*\Gr _N(\cE) \ar[r]\ar[u]^{\imath_*\beta }&0\\
		0\ar[r]&0\ar[u]\ar[r]&0\ar[u]\ar[r]&\ker\imath_*\beta \ar[u]&\\
		&0\ar[u]&0\ar[u]&0\ar[u]&\\
	}
	$$
	Since $\cF/\Gr _N(\cE)$ is supported on a finite number of points, we have $$h^0((\cF/\Gr _N(\cE)) (-H))=h^0((\cF/\Gr _N(\cE))),$$ 
which by the snake lemma implies 
	$$h^0(\ker \beta )=h^0(\ker \imath_*\beta)=h^0(\imath_*\imath ^*(\cF/\Gr _N(\cE)))=
	 h^0(\imath ^* (\cF/\Gr _N(\cE)))$$
as required.
\end{proof}

\subsection{Boundedness of sets of coherent sheaves}

Let $X$ be a scheme of finite  type over some noetherian scheme $S$.
Let $\cS$ be a set of coherent sheaves on the fibers of $X\to S$. We say that $\cS$ is \emph{bounded} 
if there exists a scheme $T$ of finite type over $S$ and a $T$-flat coherent sheaf $\cF$ on $X\times _ST/T$ such that every sheaf $\cE\in \cS$ is isomorphic to restriction of $\cF$ to one of the fibers of $X\times _ST\to T$. Note that $T$-flatness of $\cF$ is added to this definition only for convenience 
and if one can find non-flat $\cF$ then one can also find a $T$-flat one.

Let us recall the following important result of A. Grothendieck
{(see \cite[TDTE IX, 221-03 and 221-04]{FGA})}.

\begin{Proposition} \label{boundedness-FGA}
	Assume that $X$ is proper over $S$ and let $\cS$ and $\cS '$ be two bounded sets of coherent sheaves on the fibers of $X/S$. Then:
	\begin{enumerate}
		\item The set of kernels, cokernels and images of homomorphisms $\cE\to \cE'$, where $\cE\in \cS$ and $\cE'\in \cS'$, is bounded. 
		\item The set of tensor products $\cE\otimes \cE'$, where $\cE\in \cS$ and $\cE'\in \cS'$, is bounded.
		\item The set of sheaves $\cHom (\cE, \cE' )$, where $\cE\in \cS$ and $\cE'\in \cS'$, is bounded.
	\end{enumerate}
\end{Proposition}

Let us fix a normally big $k$-variety $U$ and an open embedding $j:U\hookrightarrow X$ into some small
compactification $X$.

\begin{Proposition}\label{needed-for-independence-of-compactification}
	Let $\cS$ be a bounded set of reflexive sheaves on $U/k$. Then the set $\cS ':=\{ j_*\cE\}_{\cE \in \cS}$
	is also bounded.
\end{Proposition}

\begin{proof}
	Let $T$  be a $k$-scheme of finite type and let $\cF$ be a $T$-flat coherent sheaf on 
	$U\times _kT/T$ bounding the set $\cS$. Passing to the normalization if necessary, we can assume that $T$ is normal.
	Let us set $j_T=j\times _k \id_T:U\times _kT\hookrightarrow X\times _kT$ and consider $\cG:= (j_T)_*\cF$. Since $(j_t)_*(\cF_t)$ and $\cG_t$ are isomorphic on 
	$U\otimes k(t)$, we have an isomorphism of $\cO_{X\otimes k(t)}$-modules 
	$$((j_t)_*(\cF_t))^{**}\simeq (\cG_t)^{**}.$$
	By Proposition \ref{boundedness-FGA} the set $\{(\cG_t)^{**} \}_{t\in T}$ is bounded.
	
	By assumption if  $\cE\in \cS$ then $j_*\cE$ is reflexive. So if $\cE\in \cS$ is isomorphic to $\cF_t$ for some $t\in T(k)$ then $j_*\cE$ is isomorphic to $(j_*(\cF_t))^{**}$. This shows that the set $\cS'$ is also bounded.
\end{proof}

\begin{Remark}
	The above proposition is false if $\cS$ is just a bounded set of coherent sheaves on $U/k$.
	For example, if $U=\PP^2\backslash \{x\}\hookrightarrow X=\PP^2$ and $\cS=\{\cO_{H\cap U}\}$ for some line
	$H$ passing through $x$ then $j_*(\cO_{H\cap U})$ is not a coherent $\cO_X$-module so  $\cS '=\{j_*(\cO_{H\cap U})\}$ is not a bounded set.
\end{Remark}

\medskip

The following result contained in \cite[Theorem 0.2]{La-Bog-normal} plays an important role in proofs of the main results.

\begin{Theorem} \label{boundedness}
	Let $X$ be a normal $n$-dimensional projective variety defined over an algebraically closed field $k$ of positive characteristic and let $(L_1,...,L_{n-1})$ be a collection of ample line bundles on $X$.
	Let us fix some positive integer $r$, integer $\ch_1$ and some real
	numbers $\ch _2$ and $\mu _{\max}$. 
Then the set of reflexive coherent $\cO_X$-modules $\cE$ such that 
	\begin{enumerate}
		\item $\cE$ has  rank $r$, 
		\item  $\int _X \ch_1 (\cE) L_1...L_{n-1}=\ch_1$,
		\item $\int _X \ch_2 (\cE) L_1...\widehat{L_i}...L_{n-1} \ge \ch_2$ for $i=1,...,n-1$,
		\item $\mu _{\max ,L } (\cE)\le \mu _{\max}$.
	\end{enumerate}
is bounded.
\end{Theorem}

\section{Numerically flat bundles on quasi-projective varieties}

In this section we study numerically flat bundles on possibly non-proper schemes
and introduce a new fundamental group scheme for quasi-projective varieties $X$
with $\Gamma (X, \cO_X)=k$.

\medskip

Let $X$ be a connected reduced scheme defined over an algebraically closed field $k$.
Let us recall that a vector bundle $\cE$ on $X$ is \emph{nef} if for any smooth projective curve $C$ and a morphism $f:C\to X$, every quotient of $f^*\cE$ has a non-negative degree. 
Nefness can be also defined for an arbitrary coherent $\cO_X$-modules (see, e.g., \cite[Definition 2.4]{GKPT2}). However, if $\cE$ is not locally free then this notion is not well-behaved (see, e.g., Example \ref{Kummer} or \cite[Example 8.6]{GKPT2}).
Nefness is usually defined assuming that $X$ is proper as it can have very bad properties, e.g., if $X$ contains very few proper curves.
But it is well-behaved if $X$ contains many projective curves, e.g., if $X$ is normally big.

\begin{Definition}
	We say that a vector bundle $\cE$ on $X$ is \emph{numerically flat} if both $\cE$ and $\cE^*$ are nef.
\end{Definition}

The full subcategory of the category of coherent $\cO_X$-modules, whose objects are  numerically flat vector bundles on $X$ is denoted by  $\Vect ^{\nf}(X)$.

\begin{Example}\label{example-large-num-flat}
	Let  $Z$ be a $0$-dimensional locally complete intersection subscheme of $\PP^2$.
	Then there exists a vector bundle  $\cE$ that fits into the short exact sequence
	$$0\to \cO _{\PP^2}\to \cE \to I_{Z}\to 0$$
	(see \cite[Chapter I, proof of Theorem  5.1.1]{OSS} or \cite[Chapter IV, Lemma 4.12]{Na}).
	It follows that if $X=\PP^2\backslash \Supp Z$ then 
	$\cE|_X$ is  numerically flat although $\int_{\PP^2} c_2(\cE)=\deg Z$.
	By Proposition \ref{needed-for-independence-of-compactification}, this shows that the set of numerically flat vector bundles of fixed rank on a complement of a finite number of points in $\PP^2$ is not bounded.
	Note also that $\cE$ above is pseudoeffective in the sense of \cite[Definition 2.1]{HP} and all symmetric powers of $\cE$ are slope semistable.	This shows that in \cite[Theorem 1.1]{HP} one cannot relax the assumption on stability of symmetric powers.
\end{Example}

\medskip

The following definition comes from \cite[Definition 7.1]{BHS}.

\begin{Definition}
We say that $X$ is \emph{connected by proper chains} if any two closed points of $X$ are contained in a proper connected $1$-dimensional $k$-subscheme of $X$. 
\end{Definition}

The name comes from the fact that any proper connected $1$-dimensional $k$-subscheme $C\subset  X$ 
is a chain of irreducible curves, i.e., there exists a finite set of morphisms $\nu_i: C_i\to X$ from smooth projective irreducible curves such that $C_{\reduced}=\bigcup \nu_i(C_i)$ (one can simply take the normalization of $C_{\reduced}$).

Note that if $X$ is normally big then it is connected by proper chains (see \cite[Lemma 7.3]{BHS}).

\medskip

For the definition of rigid, tensor and Tannakian categories we refer the reader to \cite{DM}.

\begin{Proposition}\label{Nori}
Assume that $X$ is connected by proper chains.
Then $\Vect ^{\nf}(X)$ is a rigid $k$-linear abelian tensor category.
\end{Proposition}

\begin{proof}
The beginning of the argument is the same as that in proof of \cite[Lemma 3.6]{No}.
Let  $\cE_1$ and $\cE_2$ be numerically flat vector bundles on $X$ and let 
	$\varphi: \cE_1\to \cE_2$ be an $\cO_X$-linear morphism. 
Let $C$ be any smooth projective curve  with  a finite morphism $f:C\to X$. Since 
$f^*\cE_1$ and $f^*\cE_2$ are semistable vector bundles of degree zero, the map	$\varphi: \cE_1\to \cE_2$
has constant rank at every point $x\in C$ and its kernel, image and cokernel are semistable of degree $0$.
  Since any two $k$-points of $X$ are contained in a chain of irreducible curves, the rank of $\varphi (x): \cE_1 (x)\to \cE_2(x)$ is constant for all $x\in X$. Therefore
$\ker \varphi$, $\im \varphi $ and $\coker \varphi $ are vector bundles and, by the above, they are also numerically flat. 
Tensor product of nef bundles is nef, so tensor product of numerically flat bundles is also numerically flat.
Also dual of a numerically flat vector bundle is numerically flat. 
This shows that $\Vect ^{\nf}(X)$ is a rigid abelian $k$-linear tensor category. 
\end{proof}

\medskip

If $X$ is connected by proper chains then any $k$-point $x$ of $X$ gives a fiber functor $T_x:\Vect ^{\nf}(X) \to k\operatorname{-}\Vect$ defined by sending $\cE$ to $\cE\otimes k(x)$. So the above proposition implies that $(\Vect ^{\nf}(X), \otimes, T_x, \cO_X)$ is a neutral Tannaka category. After writing the first version of the paper, the author noticed that this fact was first proven in \cite[Theorem 7.8]{BHS}.

\begin{Definition}
	Assume that $X$ is connected by proper chains (e.g., $X$ is normally big).
	The affine $k$-group scheme Tannaka dual to  $(\Vect ^{\nf}(X), \otimes, T_x, \cO_X)$ is called the \emph{large S-fundamental group scheme} of $X$ and denoted by $\pi_1^{LS}(X, x)$.
\end{Definition}

Note that Example \ref{example-large-num-flat} shows that if $U\subsetneq V\subsetneq \PP^2$ are big open subsets and $x\in U(k)$ then the canonical maps
$\pi_1^{LS}(U, x)\to \pi_1^{LS}(V, x)$ and $ \pi_1^{LS}(V, x) \to \pi_1^{LS}(\PP^2, x)=1$ are not isomorphisms. On the other hand, if $U\subset \PP^2_{\CC}$ is a big open subset then $\pi_1^{\rm top}(U^{\an})\simeq \pi_1^{\rm top}((\PP^2_{\CC})^{\an})=1$. This shows that even $\pi_1^{LS}(X, x)$ of a smooth complex quasi-pro\-jective variety cannot be recovered from its topological fundamental group. Note also that representations of $\pi_1^{LS}(X, x)$ of fixed rank can correspond to an unbounded set of vector bundles (see Example \ref{example-large-num-flat}). 
The author does not know any other Tannakian  group scheme associated to a category of vector bundles on 
a variety that would have this property.

\medskip

Let $X$ be a normal variety defined over an algebraically closed field $k$.
Let $\cE$ be a vector bundle on $X$ and let $D$ be a Cartier $\QQ$-divisor. Let $\pi: \PP(\cE)\to X$ denote the canonical projection and let $\xi_{\cE}=c_1(\cO_{\PP(E)}(1))$.
Let us recall that a $\QQ$-twisted vector bundle $\cE\langle D\rangle$ is called \emph{nef} if $\xi_{\cE}+\pi^*D$ is nef (see \cite[Section 6.2]{Laz}).

\begin{Definition} 
	We say that a rank $r$ vector bundle $\cE$ is \emph{projectively n-flat (projectively numerically flat)} if $\cE\langle-\frac{1}{r}c_1(\cE)\rangle$ is nef.
\end{Definition}

\begin{Remark} 
In case of smooth complex projective varieties projectively n-flat bundles were considered by N. Nakayama  in \cite[Chapter IV, Theorem 4.1]{Na}. Note that Nakayama uses a ``wrong'' notation as his  projectively flat bundles  correspond  to projective unitary representations of the topological fundamental group and not to all projective representations.
\end{Remark}

\medskip

By definition, if $X$ is connected by proper chains, then any 
numerically flat rank $r$ vector bundle on $X$ gives rise to a representation 
$\pi_1^{LS}(X, x)\to \GL _r(k)$ and to any such representation one can associate a numerically flat rank $r$ vector bundle. Similarly, any projectively n-flat rank $r$ vector bundle on $X$ gives rise to a projective representation  $\pi_1^{LS}(X, x)\to \PGL _r(k)$. 

\medskip

The following lemma is well-known (see, e.g., \cite[Proposition 7.1]{Mo}).

\begin{Lemma} \label{Q-nef}
	Let $\cE$ be a rank $r$ vector bundle on $X$.
	\begin{enumerate}
		\item 
		$\cE$ is projectively n-flat if and only if for every smooth projective curve $C$ and  a finite morphism $f:C\to X$, the bundle $f^*\cE$ is semistable.
		\item 
		$\cE$ is numerically flat if and only if for every smooth projective curve $C$ and  a finite morphism $f:C\to X$, the bundle $f^*\cE$ is semistable of degree $0$.
	\end{enumerate}
\end{Lemma}

Taking the above lemma into account, one can also define projectively n-flat vector bundles on non-normal schemes. This lemma implies the following well-known corollary (see \cite[Theorem 3.10]{FL}).

\begin{Corollary}	\label{equivalence:proj-num-flat}
Let $\cE$ be a  vector bundle on $X$. Then the following conditions are equivalent:
\begin{enumerate}
	\item 	$\cE$ is projectively n-flat, 	
	\item $\End \cE $ is numerically flat,
	\item $\End \cE $ is nef.
\end{enumerate}
\end{Corollary}

\begin{proof}
If $\cE$ is  projectively n-flat then for every smooth projective curve $C$ and  a finite morphism $f:C\to X$, the bundle $f^*\cE$ is strongly semistable. So $f^*\End \cE=\End f^*\cE$ is strongly semistable with trivial determinant. Hence $\End \cE $ is numerically flat, proving $(1)\Rightarrow (2)$.
Implication $(2)\Rightarrow (3)$ is clear.
If  $\End \cE $ is nef then for every smooth projective curve $C$ and  a finite morphism $f:C\to X$ we have
\begin{align*}
0\ge \mu _{\max} (f^*\End \cE)= \mu_{\max} (f^*\cE \otimes (f^*\cE)^*)&\ge \mu_{\max} (f^*\cE)+ \mu_{\max}((f^*\cE)^*)\\
&= \mu_{\max} (f^*\cE)-\mu_{\min} (f^*\cE)\ge 0.	
\end{align*}
Therefore  $f^*\cE$ is semistable and hence $\cE$ is projectively n-flat. 
\end{proof}

\section{Strongly numerically flat bundles on quasi-projective varieties}

 In this section we define  strongly numerically flat bundles and prove that they extend to bundles on the regular locus of a small compactification. This result implies Theorem \ref{main-num-flat}.
 We also introduce and study  the S-fundamental group scheme of normal quasi-projective varieties that admit 
 a small compactification.

\subsection{Definition of strongly numerically flat bundle}

Let $X$ be a normal variety defined over an algebraically closed field $k$.
Let us note that if $S$ is a normally big surface then by Zariski's main theorem, its small compactification $j:S\hookrightarrow \bar S$ is uniquely determined.

\begin{Definition} \label{Def:strongly-num-flat}
	A  vector bundle $\cE$ on $X$ is called \emph{strongly projectively n-flat} if it is projectively n-flat 
	and for every normally big surface $S$ with the small compactification $j:S\hookrightarrow \bar S$,
	and for any  morphism $g: S\to X$ we have
	$$\int_{\bar S}\Delta(j_*(g^*\cE))=0.$$	
	$\cE$ is called \emph{strongly numerically flat} if it is numerically flat and   for every normally big surface $S$ with the small compactification $j:S\hookrightarrow \bar S$, and any  morphism $g: S\to X$ we have
	$$\int_{\bar S}\ch_2(j_*(g^*\cE))=0.$$
\end{Definition}

The above definition makes sense in any characteristic but it is desirable to have better characterizations 
of strongly projectively (numerically) flat vector bundles.
In general, we expect that  $\cE$ is strongly projectively n-flat if and only if $\End \cE$ is strongly numerically flat (see Remark \ref{proj-flat-relation-to-num-flat}). So it is sufficient to find such a characterization for  strongly numerically flat vector bundles. The most interesting case is when $X$ can be embedded  as a big open subset into a normal projective variety (and one can even restrict to the surface case). In positive characteristic the required characterization is provided by conditions (4) and (5) in Corollary \ref{characterization-num-flat}. In general, if $X$ is normally big, one can hope that $\cE$ is strongly numerically flat if and only if both $\cE$ and $\cE^*$ are weakly positive over $X$ in the sense of Viehweg
(see \cite[2.3]{Vi2}). 

\subsection{Local freeness for projective flatness}

Let $k$ be an algebraically closed field of characteristic $p$ and 
let $X$ be a normal projective  variety of dimension $n$ defined over $k$.
The following theorem generalizes
\cite[Theorem 4.1]{La-S-fund} to normal varieties and corrects \cite[Proposition 4.5]{La-S-fund}. It also deals with projectively numerically flat bundles.  
The general strategy of proof is similar to that of \cite[Theorem 4.1]{La-S-fund} 
but the proof is more complicated due to the fact that reflexive extensions of numerically flat vector bundles on $X_{reg}$ of the same rank do not need to have the same Hilbert polynomial.

\begin{Theorem}\label{reflexive-Delta-are-loc-free} 
	Assume that $n\ge 2$  and let  $L=(L_1,...,L_{n-1})$ be a  collection of ample line bundles on $X$.
	Let  $\cE$ be a strongly slope $L$-semistable reflexive  sheaf such that 
	$$\int_X \Delta (\cE) L_1...\widehat{L_i}...L_{n-1}=0$$
	for $i=1,...,n-1$. 	Then $\cE|_{X_{reg}}$ is locally free.
	Moreover, if $0= \cE_0\subset \cE_1\subset ...\subset \cE_s=\cE$ is a
	Jordan--H\"older filtration of $\cE$ (with respect to  $L$) then  
	$$\int_X \Delta ((\cE_j/\cE_{j-1})^{**}) L_1...\widehat{L_i}...L_{n-1}= 0$$
	for $i=1,...,n-1$ and $(\cE_j/\cE_{j-1})|_{X_{reg}}$ are also locally free.
\end{Theorem}

\begin{proof}
	Possibly taking the field extension, we can assume that the base field $k$ is uncountable.
	We prove that $\cE|_{X_{reg}}$ is locally free by induction on the dimension $n$.
	Since reflexive sheaves on surfaces are locally free at regular points, the required assertion is clear for $n=2$.  So we can assume that $n\ge 3$.  For all $m\ge 0$
	we set
	$\tilde \cE_m:= F_X^{[m]}\cE (-q_mc_1(\cE ))$, where $q_m$ is as in Lemma \ref{computation-Delta}. Note that
	$$\int_X \Delta (\tilde \cE _m)L_1...\widehat{L_i}...L_{n-1}=\int_X \Delta (F_X^{[m]}\cE) L_1...\widehat{L_i}...L_{n-1}=p^{2m}\int_X \Delta (\cE)L_1...\widehat{L_i}...L_{n-1}=0$$
	for $i=1,...,n-1$.
	The proof is divided into four cases depending on stability properties of Frobenius pullbacks of $\cE$.
	
	\medskip 
	
	\emph{Case 1.} Assume that  $\cE$ is strongly slope $L$-stable, i.e., for all $m\ge 0$  the sheaves $F_X^{[m]}\cE$ are slope $L$-stable. 
	Then all $\tilde \cE_m$ are also slope $L$-stable. 	Let us fix a $k$-point $x\in X_{reg}$.
	By Theorem \ref{Bertini's-theorem}
	for $s\gg 0$  we can find a hypersurface $H\in  |L_1^{\otimes s}|$ such that if $\imath: H\hookrightarrow X$ denotes the  closed embedding then the following conditions are satisfied:
	\begin{enumerate}	
		\item $x\in H$, 
		\item $H$ is irreducible and normal, 
		\item  $H_{reg}= X_{reg}\cap H$,
		\item for all $m\ge 0$ the sheaves $\imath ^*\tilde \cE_m$ are reflexive 
		on $H\backslash \{x\}$.
	\end{enumerate} 
	By Corollary \ref{hyperplane-section-of-reflexive} each $\imath ^*\tilde \cE_m$ is a torsion free $\cO_{H}$-module 
	and \cite[Theorem 0.1]{La-Bog-normal} implies that all restrictions $\imath ^*\tilde \cE_m$ are slope  $(\imath ^*L_2,...,\imath ^*L_{n-1})$-stable. Let  $\cF_m$ be the reflexive hull of $\imath ^*\tilde \cE_m$.
	Since $\cF_m= F_H^{[m]}\cF_0 (-q_mc_1(\cF _0))$, we see that $\cF_0$ is strongly slope $(\imath ^*L_2,...,\imath ^*L_{n-1})$-stable.
	Then Bogomolov's inequality \cite[Theorem 3.4]{La-Bog-normal} and Lemma \ref{inequality-on-Delta}   imply that 
	$$0\le \int_H \Delta (\cF_0) \imath^*L_2...\widehat{\imath^*L_i}...\imath^*L_{n-1}\le s \int_X \Delta (\cE) L_1...\widehat{L_i}...L_{n-1}=0$$
	for $i=2,...,n-1$. So $\int_H \Delta (\cF_0) \imath^*L_2...\widehat{\imath^*L_i}...\imath^*L_{n-1}=0$ 	for $i=2,...,n-1$
	and by the induction assumption $\cF_0|_{H_{reg}}$ is locally free.
	By (4) the cokernel $\cG_m$ of $\imath ^*\tilde \cE _m\to \cF_m$
	is supported on $\{ x\}$. Now the argument is similar to the one from Lemma \ref{inequality-on-Delta}. Namely, 
	for every $m\ge 0$ we have a commutative diagram
	$$
	\xymatrix{
		&\imath ^*((F_X^m)^*\cE (-q_mc_1(\cE ))) \ar[r]\ar[d]^{\alpha_m}&(F_H^m)^*\cF_0 (-q_mc_1(\cF_0 )) \ar[r]\ar[d]^{\beta_m}&(F_H^m)^*\cG_0\ar[d]^{\gamma_m} \ar[r]&0\\
		0\ar[r]&\imath^* \tilde \cE_m \ar[r]&\cF_m \ar[r]&\cG_m\ar[r]&0
	}
	$$
	in which $\alpha_m$ and $\beta_m$ are isomorphisms on ${H_{reg}}$ as $F_{H_{reg}}$ and $F_{X_{reg}}$ are flat.
	It follows that $\gamma_m$ is an isomorphism and hence
	$$\chi(X, \imath_* \cG_m)= h^0 (X, \imath_* \cG_m)= h^0 (X,  \imath_* ((F_H^m)^*\cG _0))=
	p^{(n-1)m} h^0 (X,  \imath_* \cG_0).$$
	So we have
	\begin{align*}
		p^{(n-1)m} h^0 (X,  \imath_* \cG_0)= \chi (X, \imath _*\cG_{m})=&
		\chi (H,\cF_m) -\chi (X, \tilde \cE_m)+\chi (X, \tilde \cE_m(-H)).
	\end{align*}
	Since $c_1(\tilde \cE_m)= r_mc_1(\cE)$, where $0\le r_m<r$ and $\int_X \Delta (\tilde \cE _m) L_2...L_{n-1}=0$, Theorem 
	\ref{boundedness} implies that the set $\{  \tilde \cE_m\}_{m\ge 0}$ is bounded. Similarly, 
	the sets  $\{ \tilde \cE_m(-H)\}_{m\ge 0}$ and 	$\{\cF_m\} _{m\ge 0}$ are also bounded. But this shows that the right hand side of the above equalities is bounded independently of $m$. Therefore  $ h^0 (X,  \imath_* \cG_0)=0$ and $\cG_0=0$. This implies that $\imath^*\tilde \cE_0$ is locally free at $x$ and hence $\tilde \cE_0$ is locally free at $x$ (cf. \cite[Lemma 1.14]{La-JEMS}). This shows that in the  first case $\cE|_{X_{reg}}$ is locally free.

	\medskip
	
	\emph{Case 2.} Let us assume that $\cE$ is not slope $L$-stable
	but all Jordan--H\"older quotients of $\cE$ are  strongly slope $L$-stable.
	Let $0= \cE_0\subset \cE_1\subset ...\subset \cE_s=\cE$ be a
	Jordan--H\"older filtration of $\cE$ and let us set $\cF_j:=(\cE_j/\cE_{j-1})^{**}$
	and $r_j=\rk \cF_j$.
	By Bogomolov's inequality for strongly slope semistable sheaves (see \cite[Theorem 3.4]{La-Bog-normal}) and by \cite[Lemma 2.5]{La-Bog-normal} we have
	\begin{align*}
		0=\frac{\int _X \Delta  (\cE)L_1...\widehat{L_i}...L_{n-1}}{r}\ge \sum _j \frac{\int _X \Delta  (\cF_j)L_1...\widehat{L_i}...L_{n-1}}{r_j}\ge 0.
	\end{align*}
So $\int _X \Delta  (\cF_j)L_1...\widehat{L_i}...L_{n-1}=0$
for all $j$ and $i=1,...,n-1$.
Since all $\cF_j $ are strongly slope $(L_1,...,L_{n-1})$-stable, the first case implies that all $\cF_j|_{X_{reg}}$ are locally free. Hence
	Lemma \ref{Chern-classes-filtrations} shows that $\cE |_{X_{reg}}$ and $(\cE_j/\cE_{j-1})|_{X_{reg}}$ are locally free.

	\medskip
	
	\emph{Case 3.} Assume that  $\cE$ is  slope $L$-stable but it
	is not strongly slope $L$-stable. In this case there exists $m> 0$
	such that $F_X^{[m]}\cE$ is not slope $L$-stable but all Jordan--H\"older quotients of $F_X^{[m]}\cE$ are  strongly slope $L$-stable. So by the second case $(F_X^{[m]}\cE)|_{X_{reg}}=( F_{X_{reg}}^m)^*(\cE|_{X_{reg}})$ is locally free. But $F_{X_{reg}}$ is flat, so $\cE|_{X_{reg}}$ is also locally free.

	\medskip
	
	\emph{Case 4.} Now assume that $\cE$ is strictly slope
	$L$-semistable. Choose an arbitrary Jordan--H\"older filtration
	and set $\cF_j:=(\cE_j/\cE_{j-1})^{**}$.
	Each $\cF_j$ is slope $L$-stable and strongly slope
	$L$-semistable. As in Case 2, the discriminant inequalities
	are equalities, so all $\cF_j$ satisfy the required
	discriminant vanishings. By Cases 1 and 3, their restrictions
	to $X_{\rm reg}$ are locally free. Applying Lemma~2.3
	exactly as in Case 2 proves that $\cE|_{X_{\rm reg}}$ and
	all $(\cE_j/\cE_{j-1})|_{X_{\rm reg}}$ are locally free.
	Since the filtration was arbitrary, this also proves
	the last assertion.
\end{proof}

\medskip

\begin{Remark}
	In the above theorem we do not claim that the quotients  $\cE_j/\cE_{j-1}$
	are reflexive for all $j$.  
\end{Remark}

\medskip

\begin{Corollary}\label{characterization-proj-flat}
	Let $\cE$ be a coherent reflexive $\cO_X$-module.
	Then the following conditions are equivalent:
	\begin{enumerate}
		\item $\cE|_{X_{reg}}$ is a strongly projectively n-flat vector bundle. 
		\item For some big open subset $U\subset X$, the restriction $\cE|_{U}$ is a strongly projectively n-flat vector bundle.
		\item $\int_X \Delta  (\cE)=0$ and  for every collection $L=(L_1,...,L_{n-1})$ of nef 
		line bundles, $\cE$ is strongly slope $L$-semistable.
		\item  For some collection $L=(L_1,...,L_{n-1})$ of ample line bundles,
		$\cE$ is strongly slope $L$-semistable and
		$$\int_X \Delta (\cE) L_1...\widehat{L_i}...L_{n-1}=0$$
		for $i=1,...,n-1$.
	\end{enumerate}	
\end{Corollary}

\begin{proof}
Implications  $(1)\Rightarrow (2)$ and $(3)\Rightarrow (4)$  are clear.
To see that $(2)\Rightarrow (3)$ note that
$\cE$ is strongly slope $(L_1,...,L_{n-1})$-semistable follows from Lemma \ref{Q-nef}.
The condition on the discriminant of $\cE$  follows from the corresponding condition in the definition of a  strongly projectively n-flat bundle by restricting to a very general complete intersection surface.

If $\cE$ satisfies (4) then $\cE|_{X_{reg}}$ is locally free by Theorem \ref{reflexive-Delta-are-loc-free}.
The argument showing that $\cE|_{X_{reg}}$ is projectively n-flat is similar to that from \cite[Proposition 5.1]{La-S-fund}. We use notation from the proof of Theorem \ref{reflexive-Delta-are-loc-free}. Since  the set $\{  \tilde \cE_m\}_{m\ge 0}$ is bounded,
there exists an ample line bundle $M$ on $X$ such that $\tilde \cE _m\otimes M$ is globally
generated for all $m\ge 0$. Therefore for any smooth projective
curve $C$ and a finite morphism $f:C\to X_{reg}$, the bundles
$$f^*(\tilde \cE_m\otimes M)= (F_C^{m})^*(f^*\cE) \otimes \det (f^*\cE)^{-q_m} \otimes f^*M$$ are globally generated. 
In particular,
$\mu_{\min}( f^*(\tilde \cE_m\otimes M))\ge 0$ shows that
$$  p^m \mu_{\min} (f^*\cE) \ge \mu_{\min}((F_C^{m})^*(f^*\cE))\ge  q_m\deg f^*\cE -\deg f^*M $$
for all $m\ge 0$.
Dividing by $p^m$ and passing with $m$ to infinity we get 
$$\mu _{\min } (f^*\cE)\ge \frac{1}{r}\deg f^*\cE=\mu (f^*\cE).$$ 
Therefore $f^*\cE$ is semistable and hence
$\cE|_{X_{reg}}$ is projectively n-flat by Lemma \ref{Q-nef}. 
Now let $S$ be a normally big surface with small compactification $j:S\hookrightarrow \bar S$ and 
let  $g: S\to X_{reg}$ be a morphism. Then the set $\{  g^*(\tilde \cE_m|_{X_{reg}}) \}_{m\ge 0}$ is bounded
and hence by Proposition \ref{needed-for-independence-of-compactification} the set $\{  j_*g^*(\tilde \cE_m|_{X_{reg}}) \}_{m\ge 0}$ is also bounded. Since Euler characteristic is locally constant in flat families of sheaves, the set
 $\{ \chi(\bar S, j_*g^*(\tilde \cE_m|_{X_{reg}} ))\}_{m\ge 0}$ is finite. Let us set $\tilde \cE:=j_*g^*(\cE|_{X_{reg}})$.
Note that 
 $$  F_{\bar S}^{[m]}\tilde \cE (-q_mc_1(\tilde \cE))=j_*g^*(\tilde \cE_m|_{X_{reg}})$$
 as both sheaves are reflexive and they agree on $S$. Therefore
$\int_{\bar S}\Delta(\tilde \cE)=0$ by Lemma \ref{computation-Delta}.
This shows that $(4)\Rightarrow (1)$.
\end{proof}

\subsection{Local freeness for numerical flatness}
 
 \begin{Lemma}\label{num-flat-easy-consequences}
 	Let $\cE$ be a coherent reflexive $\cO_X$-module. If the set  $\{ F_X^{[m]} \cE \} _{m\ge 0}$ is bounded
 	then the class of $c_1(\cE)$ in $ \bar B^1(X)$ vanishes. Moreover,  for every 
 	collection of line bundles $(L_1,...,L_{n-2})$ we have
 	$\int _X c _2 (\cE)L_1...L_{n-2}=0.$
 \end{Lemma}
 
 \begin{proof}
 	By definition there exists a $k$-scheme $S$ of finite type and an $S$-flat $\cO_{X\times _kS}$-module $\cF$ such that for every $m\ge 0$ there exists a $k$-point $s_m\in S$ such that $\cF_{s_m}\simeq  F_X^{[m]} \cE $. Note that for every two $k$-points $s, s'$ of $ S$ lying in the same connected component we have $[c_1(\cF_s)]=[c_1(\cF_{s'})]\in B^1(X)$. 
 	Since $c_1( F_X^{[m]} \cE) = p^m c_1(\cE)$ and $S$ is of finite type over $k$, the set $\{ [p^m c_1(\cE )] \} _{m\ge 0}$ of corresponding classes in $B^1(X)$ is finite.  But $ \bar B^1(X)$ is a free $\ZZ$-module
 	(see \cite[Lemma 1.20]{La-Bog-normal}), so the class $[c_1(\cE)]\in  \bar B^1(X)$ vanishes.
 	
 	To prove the second assertion note that for every collection  $(L_1,...,L_j)$ of line bundles on $X$ the function $S\to \ZZ$ given by \small
 	$${ s\to \chi (X\otimes k(s), c_1(L_1)...c_1(L_j)\cdot \cF_s)=\sum _{\sigma : \{1,...,j\}\to \{ 0,1\}}(-1)^{|\sigma ^{-1}(1)|}\chi (X\otimes k(s), \cF_s\otimes \bigotimes _{i=1}^j L_i^{-\sigma (i)})}$$ \normalsize
 	is locally constant  (cf.~\cite[Chapter VI, Definition 2.6]{Ko}). This follows from the fact that 
 	$\cF\otimes \bigotimes _{i=1}^j L_i^{-\sigma (i)}$ is $S$-flat and
 	Euler characteristic is locally constant in flat families of sheaves. Since $S$ is of finite type over $k$, this implies that the set $\{ \chi (X, c_1(L_1)...c_1(L_j)\cdot  F_X^{[m]} \cE )\}_{m\ge 0}$ is finite.
 	In particular, for  every collection of line bundles $(L_1,...,L_{n-2})$ we have
 	$$\int _X \ch _2 (\cE)L_1...L_{n-2}=
 	\lim _{m\to \infty }\frac {  \chi (X, c_1(L_1)...c_1(L_{n-2}) \cdot {F_X^{[m]}\cE}  ) }{p^{2m}}=0.$$
 	Since $[c_1(\cE)]=0\in  \bar B^1(X)$, we get the required vanishing.
 \end{proof}

\medskip

\begin{Corollary}\label{characterization-num-flat}
	Let $\cE$ be a coherent reflexive $\cO_X$-module.
	Then the following conditions are equivalent:
	\begin{enumerate}
		\item  $\cE|_{X_{reg}}$ is a numerically flat vector bundle and
		 for every normally big surface $S$  and any morphism $g: S\to X_{reg}$ we have $[c_1(j_*(g^*\cE))]= 0\in \bar B^1(\bar S)$ and 
		$$\int_{\bar S}c_2(j_*(g^*\cE))=0,$$
		where  $j:S\hookrightarrow \bar S$ is the small compactification of $S$.
		\item $\cE|_{X_{reg}}$ is a strongly numerically flat vector bundle. 
		\item For some big open subset $U\subset X$, the restriction $\cE|_{U}$ is a strongly numerically flat vector bundle.
		\item  $\cE|_{X_{reg}}$ is locally free and	the set  $\{ (F_{X_{reg}}^{m})^* (\cE|_{X_{reg}}) \} _{m\ge 0}$ is bounded.
		\item For some big open subset $U\subset X$, the restriction $\cE|_{U}$ is
		 locally free and	the set  $\{ (F_{U}^{m})^* (\cE|_{U}) \} _{m\ge 0}$ is bounded.
		\item The set   $\{ F_X^{[m]} \cE \} _{m\ge 0}$ is bounded.
		\item  The class $[c_1(\cE)]=0\in  \bar B^1(X)$, for every collection $L=(L_1,...,L_{n-1})$ of nef 
		line bundles 
		$$\int_X c_2 (\cE) L_1...\widehat{L_i}...L_{n-1}= 0$$
		for $i=1,...,n-1$, and $\cE$ is strongly slope $L$-semistable.
		\item  For some collection $L=(L_1,...,L_{n-1})$ of ample line bundles,
		$\cE$ is strongly slope $L$-semistable, $\int_X \ch_1 (\cE) L_1...L_{n-1}=0$ and
		$$\int_X \ch_2 (\cE) L_1...\widehat{L_i}...L_{n-1}\ge 0$$
		for $i=1,...,n-1$.
	\end{enumerate}	
\end{Corollary}

\begin{proof}
Implications  $(1)\Rightarrow (2)$ and $(2)\Rightarrow (3)$ are clear. If (3) holds then for any collection $(L_1,...,L_{n-1})$ of ample line bundles we have  
$$\int_X \ch_1 (\cE) L_1...L_{n-1}\ge 0$$
and 
$$\int_X \ch_1 (\cE^*) L_1...L_{n-1}=-\int_X \ch_1 (\cE) L_1...L_{n-1}\ge 0.$$	
So 	$\int_X \ch_1 (\cE) L_1...L_{n-1}=0$. Restricting to a very general complete intersection surface, our assumption implies that
$$\int_X \ch_2 (\cE) L_1...\widehat{L_i}...L_{n-1}= 0$$
for $i=1,...,n-1$. Since semistability can be checked on big open subsets, we see that all $ F_X^{[m]} \cE $ are slope $L$-semistable. This shows implication  $(3)\Rightarrow (8)$.
The implication  $(4)\Rightarrow (5)$ is clear and  $(5)\Rightarrow (6)$ follows from Proposition \ref{needed-for-independence-of-compactification}. If (6) is satisfied then 	by Lemma \ref{num-flat-easy-consequences} we have
	$[c_1(\cE)]=0\in  \bar B^1(X)$ and $\int_X \ch_2 (\cE) L_1...\widehat{L_i}...L_{n-1}= 0$
	for $i=1,...,n-1$. Note that for every $m$ the sheaf $F^{[m]}_X\cE$ also satisfies (6).
	So to prove $(6)\Rightarrow (7)$ it is sufficient to show that  $\cE$ is slope $L$-semistable. Let $\cF\subset \cE$ be a saturated $\cO_X$-submodule. 
	If $\mu(\cF)>\mu (\cE)=0$ then the set $\{ F_X^{[m]} \cF \} _{m\ge 0}$ is not bounded
	as $\mu( F_X^{[m]} \cF)=p^m \mu(\cF)$. This contradicts the fact that 
	$\{ F_X^{[m]} \cE \} _{m\ge 0}$ is bounded.
	Implication $(7)\Rightarrow (8)$ is obvious. 
	Now let us prove that $(8)\Rightarrow (1)+(4)$. Since
	$$\int_X \ch_2 (F_X^{[m]} \cE) L_1...\widehat{L_i}...L_{n-1}= p^{2m}\int_X \ch_2 (\cE) L_1...\widehat{L_i}...L_{n-1}\ge 0,$$
	the set $\{ F_X^{[m]} \cE \} _{m\ge 0}$ is bounded by Theorem \ref{boundedness}. 
	So the set $\{ (F_{X_{reg}}^{m})^* (\cE|_{X_{reg}}) \} _{m\ge 0}$ is also bounded. By Lemma \ref{num-flat-easy-consequences}, $[c_1(\cE)]=0\in  \bar B^1(X)$ and
	$\int_X \Delta (\cE) L_2...L_{n-1}=0$. So Theorem \ref{reflexive-Delta-are-loc-free}
	shows that  $\cE|_{X_{reg}}$ is locally free and $\cE$ satisfies (4).	
By Corollary \ref{characterization-proj-flat} we see that $\cE|_{X_{reg}}$ is projectively n-flat. 
Since $[c_1(\cE)]=0\in  \bar B^1(X)$,   $\cE|_{X_{reg}}$ is also numerically flat. Now let $S$ be a normally big surface with small compactification $j:S\hookrightarrow \bar S$ and a morphism $g: S\to X_{reg}$. Then the set  $\{ (F_{S}^{m})^* (g^*\cE) \} _{m\ge 0}$ is bounded. So by implication $(5)\Rightarrow (7)$, we have  $[c_1(j_*(g^*\cE))]=0\in   \bar B^1(\bar S)$ and
$$\int_{\bar S}c_2(j_*(g^*\cE))=0.$$
\end{proof}

\medskip

\begin{Remark}
Assume that $X$ is smooth projective defined over an algebraically closed field $k$ of characteristic $p>0$
and $\cE$ is locally free. Then any numerically flat vector bundle on $X$  is strongly numerically flat.
The fact that conditions (1) and (6) are equivalent follows in this case from the proof of \cite[Proposition 5.1]{La-S-fund}.  \cite[Theorem 4.4]{La-AnnMath} and \cite[Theorem 4.1]{La-S-fund} give equivalence of (4) and (7) for one ample polarization.
\end{Remark}

\begin{Remark}
	If $\cE$ is a numerically flat vector bundle on  a proper scheme over a field then  all Chern classes of $\cE$ are numerically trivial (see \cite[Proposition 2.3]{FL}). In positive characteristic this follows from condition (6) in Corollary \ref{characterization-num-flat}.
\end{Remark}

\medskip

The following example shows that Corollary \ref{characterization-num-flat} is rather subtle.

\begin{Example}
	Let $\bar X$ be a normal projective variety and let $j: X\hookrightarrow \bar X$ be its regular locus. Assume that on $X$ we have a non-trivial Cartier divisor $D$ such that $2D\sim 0$. Then $\cE=\cO_X\oplus \cO_X(D)$ is a strongly numerically flat vector bundle on $X$. Let $\bar D$ be the Weil divisor obtained by the closure of $D$ in $\bar X$.  This divisor is usually non-Cartier but it is $2$-Cartier and $j_*\cE= \cO_{\bar X} \oplus \cO_{\bar X} (\bar D)$. Then it is easy to see that  
	$\int_{\bar X} \ch_1 (j_*\cE)=\int_{\bar X} \ch_2 (j_*\cE)= 0.$ If $p>2$, the reflexive pullback of  $j_*\cE$ is trivial on the double covering of $\bar X$ associated to $\bar D$ (this covering is branched at the points where $\bar D$ is not Cartier). However, if $\cE$ is a non-trivial extension of $\cO_X(D)$ by $\cO_X$ then usually $\int_{\bar X} \ch_2 (j_*\cE)\ne  0.$ In fact, Example \ref{example-large-num-flat} shows that such extensions are not strongly numerically flat even if $\bar X$ is smooth and $D=0$.	
\end{Example}

The following example shows that in Theorem \ref{main-num-flat} we cannot replace strong semistability with semistability even if $X=\bar X$.

\begin{Example}
Let $C$ be a  smooth projective curve (of genus $\ge 2$) defined over an algebraically closed field of positive characteristic.  Let $\cF$ be any  semistable vector bundle of degree $0$ on $C$, which is not strongly semistable. Let $X=\PP^1\times C$ and let $\cE$ be the pullback of $\cF$ via the canonical projection. Then $\cE$ is a semistable vector bundle with vanishing Chern numbers. In this case it is easy to see that  	the set  $\{ (F_{X}^{m})^* \cE \} _{m\ge 0}$ is not bounded
since  slopes of the maximal destabilizing subbundles of $ (F_{X}^{m})^* \cE $ tend to infinity.
\end{Example}

\medskip

\begin{Corollary}
Let $\cE$ be a coherent reflexive $\cO_X$-module. If $\cE|_{X_{reg}}$ is a strongly projectively n-flat vector bundle then  $\End \cE|_{X_{reg}}$ is strongly numerically flat.
\end{Corollary}

\begin{proof}	
The proof of Corollary \ref{characterization-proj-flat} shows that $\cE|_{X_{reg}}$ is a strongly projectively n-flat vector bundle if and only if the set $\{ \tilde \cE_m|_{X_{reg}} \}_{m\ge 0}$ defined in that proof is bounded. By Proposition \ref{boundedness-FGA} the set  $\{ \End \tilde \cE_m|_{X_{reg}} \}_{m\ge 0}$ is also bounded. But $$\End \tilde \cE_m= F_X^{[m]} (\End \cE)$$
as both sheaves are reflexive and they agree on $X_{reg}$. So the assertion follows from Corollary \ref{characterization-num-flat}. 
\end{proof}

\begin{Remark}\label{proj-flat-relation-to-num-flat}
We expect that in the above corollary we have not only implication but also equivalence. This holds if  for any reflexive sheaf $\cF$ on a normal projective surface $S$ we have 
	$$\int _S \Delta (\cF)=\int_Sc_2 (\End \cF)$$
	as in the vector bundle case. Unfortunately, this equality is not known in general.
\end{Remark}

\medskip

\begin{Proposition}\label{reflexive-num-flat-are-loc-free}
	Assume that $n\ge 2$  and let  $L=(L_1,...,L_{n-1})$ be a  collection of ample line bundles on $X$.
	Let $\cE$ satisfy one of the equivalent conditions from Corollary \ref{characterization-num-flat}
	and let  $0= \cE_0\subset \cE_1\subset ...\subset \cE_s=\cE$ be a
	Jordan--H\"older filtration of $\cE$ with respect to  $L$. Then  all
	$(\cE_j/\cE_{j-1})|_{X_{reg}}$ are strongly numerically flat vector bundles.
\end{Proposition}

\begin{proof}
	By Theorem \ref{reflexive-Delta-are-loc-free} all $(\cE_j/\cE_{j-1})|_{X_{reg}}$ are locally free and for $i=1,...,n-1$ we have
	$$\int _X \Delta (\cF _j) L_1...\widehat{L_i}...L_{n-1} =0,$$ 
	where $\cF_j:=(\cE_j/\cE_{j-1})^{**}$.
	By the Hodge index theorem  
	$${\int _X c_1 (\cF _j)^2 L_1...\widehat{L_i}...L_{n-1}}
	\le \frac {(\int _X c_1 (\cF _j) L_1...L_{n-1} )^2}{L_1...L_i^2...L_{n-1}}=0.$$
	So \cite[Lemma 2.4]{La-Bog-normal} implies that
	\begin{align*}
		0\le \int _X \ch _2 (\cE) L_1...\widehat{L_i}...L_{n-1} \le \sum _j\int _X \ch _2 (\cF _j) L_1...\widehat{L_i}...L_{n-1}
		\le  
		\sum _j \frac {\int _X c_1 (\cF _j)^2 L_1...\widehat{L_i}...L_{n-1}}{2r_j}\le 0,
	\end{align*}
where $r_j=\rk \cF_j$.
	Therefore we get equalities
	$$\int _X \ch _2 (\cF _j) L_1...\widehat{L_i}...L_{n-1}=0$$  
	for all $i$ and $j$. Since $\cF_j $ is strongly slope $L$-semistable, it satisfies condition (8) from Corollary \ref{characterization-num-flat}. So $(\cE_j/\cE_{j-1})|_{X_{reg}}=\cF_j|_{X_{reg}}$
	is strongly numerically flat.
\end{proof}

\subsection{S-fundamental group scheme for quasi-projective varieties}

Let $X$ be a normal quasi-projective variety defined over an algebraically closed field of positive characteristic $p$. Assume that $X$ admits a normal projective compactification  $j:X\hookrightarrow \bar X$
such that the complement of $X$ in $\bar X$ has codimension $\ge 2$.
We denote by $\Vect ^{\snf}(X)$ the full subcategory of the category of coherent $\cO_X$-modules, whose objects are strongly numerically flat vector bundles on $X$.

In view of Corollary \ref{characterization-num-flat}, we say that a coherent reflexive $\cO_{\bar X}$-module is \emph{strongly numerically flat} if the set   $\{ F_{\bar X}^{[m]} \cE \} _{m\ge 0}$ is bounded.
The  full subcategory of the category of coherent $\cO_{\bar X}$-modules, whose objects are strongly numerically flat reflexive $\cO_{\bar X}$-modules is denoted by  $\Ref ^{\snf}(\bar X)$.

\begin{Example}\label{Kummer}
Let $\bar X$ be a singular Kummer surface, i.e., the quotient of an abelian surface by the sign involution (for simplicity assume that the characteristic of the base field  $p\ne 2$). 
Since the map $\pi: A\to \bar X$ is \'etale of degree $2$ over $X:={\bar X}_{reg}$, there exists a $2$-torsion Weil divisor $D$ on ${\bar X}$ such that $\pi_*\cO_A=\cO_{\bar X}\oplus \cO_{\bar X}(D)$. Then  $\pi^{[*]}\cO_{\bar X}(D)=\cO_A$ and $\cE=\cO_{\bar X}(D)$ is a strongly numerically flat reflexive $\cO_{\bar X}$-module. Note that both $\cE$ and $\cE ^*$ are weakly positive (see \cite[2.3]{Vi2}) over $X$. However, $\cO_{\PP(\cE)}(1)$ is not nef. Indeed, if $C$ is any smooth curve on ${\bar X}$ passing through one of the singular points  then $\cO_{\PP(\cE|_C)}(1)$ is not nef as $\cE|_C$ splits into a direct sum of a torsion sheaf and a negative line bundle (cf. \cite[Remark 2.6]{Ga} for a different argument). 
\end{Example}

\medskip

\begin{Proposition}\label{Vect-snf-rigid-abelian}
	$\Vect ^{\snf}(X)$ is a rigid  abelian $k$-linear tensor category.
\end{Proposition}

\begin{proof}
	Let  $\cE_1$ and $\cE_2$ be strongly numerically flat reflexive sheaves on $\bar X$ such that their restriction to $X$ is locally free and let  
	$\varphi: \cE_1\to \cE_2$ be an $\cO_{\bar X}$-linear morphism. 
	By Proposition \ref{boundedness-FGA} we know that the sets 
	$$\{F_{\bar X}^{[m]} (\ker \varphi) \}_{m\ge 0},\quad \{F_{\bar X}^{[m]} ((\im \varphi)^{**}) \}_{m\ge 0}\quad  \hbox{ and }\quad \{F_{\bar X}^{[m]} ((\coker \varphi )^{**}) \}_{m\ge 0}$$	
	are bounded. Therefore $\ker \varphi$, $(\im \varphi )^{**}$ and $(\coker \varphi )^{**}$ are strongly numerically flat. Similarly, since 
	$$F_{\bar X}^{[m]}(\cE_1\hat \otimes \cE_2)=  \cHom _{\bar X} ( \cHom _{\bar X} (F_{\bar X}^{[m]}\cE_1\otimes _{\cO_{\bar X}}F_{\bar X}^{[m]}\cE_2, \cO_{\bar X}), \cO_{\bar X}),$$
	Proposition \ref{boundedness-FGA} implies that the set  $\{ F_{\bar X}^{[m]}(\cE_1\hat \otimes _{\cO_{\bar X}} \cE_2)	\}_{m\ge 0}$ is bounded and hence $\cE_1\hat \otimes  \cE_2$ is strongly numerically flat. Since dual of a strongly numerically flat reflexive sheaf is strongly numerically flat, also $\cHom_{\bar X} (\cE_1, \cE_2)= \cE_1^*\hat \otimes  \cE_2$ is strongly numerically flat. By Proposition \ref{Nori}, the kernel, image and cokernel of
	$\varphi|_X$ are locally free. Hence they coincide with the
	restrictions of $\ker\varphi$, $(\im\varphi)^{**}$ and
	$(\coker\varphi)^{**}$, respectively, and are strongly numerically
	flat. Together with closure under tensor products and duals,
	this proves that $\Vect^{\snf}(X)$ is a rigid abelian
	$k$-linear tensor category.
\end{proof}

\begin{Remark}
	Note that the above proposition does not follow from Proposition \ref{Nori}. For example,
	it is not clear if the second Chern class of $\cE_1\hat\otimes \cE_2$ can be computed using Chern classes of $\cE_1$ and $\cE_2$.  The above proposition says that it is indeed so for strongly numerically flat reflexive $\cO_{\bar X}$-modules. An even more serious problem is that the second Chern class of the kernel of a map between vector bundles depends on the map and it cannot be computed from Chern classes of the bundles. 
\end{Remark}

\medskip

Corollary \ref{characterization-num-flat} implies the following result:

\begin{Lemma}\label{equivalences-of-num-flat-categories}
If $X$ is smooth then the functors
$j_* :\Vect ^{\snf}(X)\to\Ref ^{\snf}({\bar X}) $ and $j^*:\Ref ^{\snf}({\bar X})\to \Vect ^{\snf}(X)$ define adjoint equivalences of the categories $\Vect ^{\snf}(X)$ and $ \Ref ^{\snf}({\bar X})$.
In particular, the open embedding $X\hookrightarrow {\bar X}_{reg}$ induces an equivalence of categories
 $\Vect ^{\snf}(X)$ and $\Vect ^{\snf}({\bar X}_{reg})$.
\end{Lemma}

In general, restriction to ${\bar X}_{reg}$ gives an equivalence of $\Vect ^{\snf}(X)$ with the full subcategory of 
$\Vect ^{\snf}({\bar X}_{reg})$. This subcategory is also a rigid  abelian $k$-linear tensor category.
Every $k$-point $x$ of $X$ gives a fiber functor $T_x:\Vect ^{\snf}(X) \to k\operatorname{-}\Vect$ defined by sending $\cE$ to $\cE\otimes k(x)$. Then  $(\Vect ^{\snf}(X), \otimes, T_x, \cO_X)$ is a neutral Tannaka category
by Proposition \ref{Vect-snf-rigid-abelian}.

\begin{Definition}
	The affine $k$-group scheme Tannaka dual to  $(\Vect ^{\snf}(X), \otimes, T_x, \cO_X)$ is called the \emph{S-funda\-men\-tal group scheme} of $X$ and it is denoted by $\pi_1^S(X, x)$.
\end{Definition}

As in the projective case, the above S-fundamental group scheme allows us to recover Nori's fundamental group scheme $\pi_1^N(X, x)$ (cf. \cite[Section 6]{La-S-fund}). In particular, 
if $k$ is an algebraic closure of a finite field then $\pi_1^S(X, x)=\pi_1^N(X, x)$.
Existence of the S-fundamental group scheme for normally big varieties raises many questions that we do not study in this paper. Let us just mention that the Mal\v{c}ev--Grothendieck theorem 
suggests that vanishing of $\pi_1^N(X, x)$ implies vanishing of $\pi_1^S(X, x)$. This is indeed the case if
$X$ is projective (see \cite{EM}). An analogous problem relating $\cO_X$-coherent ${\mathcal D}_X$-modules to 
the \'etale fundamental group was considered in \cite{ES}.

\begin{Remark}
 Simple representations of  $\pi_1^S(X, x)$ correspond to simple objects in $\Vect ^{\snf}(X)$. If 
 $X$ is smooth and  $(L_1,...,L_{n-1})$ is a   collection  of ample line bundles on $\bar X$ then Proposition \ref{reflexive-num-flat-are-loc-free} implies that $\cE \in \Vect ^{\snf}(X)$ is simple if and only if $j_*\cE$ is slope $(L_1,...,L_{n-1})$-stable.	In particular, this condition does not depend on the choice of a small compactification of $X$ and an ample collection.
\end{Remark}

The following proposition generalizes \cite[Chapter II, the first Corollary after Proposition 6 and Proposition 7]{No} to our setting.

\begin{Proposition}
	Let $ X \hookrightarrow  X'\hookrightarrow \bar X$ be any open subset factoring $j$
and let $x\in X(k)$. Then we have an induced faithfully flat homomorphism
$$\pi_1^S(X, x)\to \pi_1^S(X', x)$$
of S-fundamental group schemes, which is an isomorphism if $X'$ is regular. In particular, if $X$
is smooth then the open embedding $X\hookrightarrow \bar X_{reg}$ induces an isomorphism $\pi_1^S(X, x)\simeq \pi_1^S(\bar X_{reg}, x)$.
\end{Proposition}
 
\begin{proof}
The assertions follow easily from 
Lemma \ref{equivalences-of-num-flat-categories} and \cite[Proposition 2.21]{DM}.
\end{proof}

\begin{Corollary}\label{asked-by-referee}
Let  $Y$ and $Y'$ be normal projective varieties that are isomorphic in codimension $1$. Then $\pi_1^S(Y_{reg})\simeq \pi_1^S(Y'_{reg})$.
\end{Corollary}

\begin{proof}
	Let $U\subset Y$ and $U'\subset Y'$ be open subsets with complements of codimension $\ge 2$ 
	and let $\varphi: U\to U'$ be an isomorphism. Then by the above proposition we have the following isomorphisms:
$$\pi_1^S(Y_{reg})\simeq \pi_1^S(U_{reg})\mathop{\DOTSB\relbar\joinrel\relbar\joinrel\relbar\joinrel\longrightarrow}^{(\varphi|_{U_{reg}})_*}\pi_1^S(U'_{reg})\simeq \pi_1^S(Y'_{reg}).$$
\end{proof}

\section{Local freeness for modules with $\lambda$-connection}

In this section we prove Theorem \ref{main-Higgs} and its generalization to multi-polarizations (see Remark \ref{proof-ofmain-Higgs}). Before stating the theorem let us recall the following special case of \cite[Definition 1.5]{La-Bog-normal}.

\begin{Definition}
Let $U$ be a normally big variety with a small compactification $X$.
Assume that $\tilde U$ is a lifting of $U$ to $W_2(k)$. We say that $X$ has \emph{$F$-liftable singularities in codimension $2$ with respect to $\tilde U$} if there exists a closed subset $Z\subset X$ of codimension $\ge 3$ 
such that every point $x\in X\backslash Z$ has an open neighbourhood $U_x\subset X\backslash Z$
for which there exists a lifting $\tilde U_x$ of $U_x$ and a lifting $F_{\tilde U_x}: \tilde U_x\to \tilde U_x$ of the Frobenius morphism $F_{U_x}: U_x\to U_x$ such that $(U_x\cap U, \cO_{\tilde U_x}|_{U_x\cap U})$
is isomorphic to $(U_x\cap U, \cO_{\tilde U}|_{U_x\cap U})$. 
\end{Definition}

\medskip

\begin{Theorem} \label{reflexive-loc-free-Delta-Higgs}
	Let $U$ be a normally big variety of dimension $n$ defined over an algebraically closed field $k$ of characteristic $p>0$.
	Assume that $U$ has a lifting $\tilde U$ to $W_2(k)$. Let $X$ be a small compactification of $U$ such that $X$ has $F$-liftable singularities in codimension $2$ with respect to $\tilde U$ and let $L=(L_1,...,L_{n-1})$
	be a collection of ample line bundles on $X$.
	Let  $(V, \nabla)$ be a reflexive $\cO_X$-module with an integrable $\lambda$-connection for some $\lambda\in k$. Assume that  $(V, \nabla)$ is  slope $L$-semistable of rank
	$r\le p$ with
	$$\int_X \Delta (V) L_1...\widehat{L_i}...L_{n-1}= 0$$
	for $i=1,...,n-1$. Let $N^{\bullet }$ be a Griffiths transverse filtration on $(V, \nabla)$ with slope $L$-semistable
associated graded Higgs sheaf  $(\cE:=\Gr _{N} V, \theta )$. Then the following conditions are satisfied:
	\begin{enumerate}
		\item if $M=(M_1,..., M_{n-1})$ is a collection of nef line bundles on $X$ then 
		both $(V, \nabla)$ and $(\cE , \theta)$ are slope $M$-semistable,
		\item $\int _{X}\Delta (V)= \int _{X}\Delta (\cE^{**})= 0$,
		\item  $V|_{X_{reg}}$ and  $\cE|_{X_{reg}}$ are locally free.
	\end{enumerate}
\end{Theorem}

\begin{proof}
	By \cite[Theorem 1.13]{La-Bog-normal} we can assume that $U=X_{reg}$.	
	Let	us set $(V_{-1}, \nabla_{-1}):=(V, \nabla)$ and $S_{-1}^{\bullet}:=N^{\bullet }$. Let
	$${  \xymatrix{
			(V_{-1}, \nabla _{-1}) \ar[rd]^{(\Gr _{S_{-1}})^{**}}	&& (V_0, \nabla _0)\ar[rd]^{(\Gr _{S_0})^{**}}&& (V_1, \nabla _1)\ar[rd]^{(\Gr _{S_1})^{**}}&\\
			&	(\cE_0, \theta _0)\ar[ru]^{C_{\tilde U}^{[-1]}}&&	(\cE_1, \theta_1)\ar[ru]^{C_{\tilde U}^{[-1]}}&&...\\
	}}$$
	be  a Higgs--de Rham sequence, i.e.,  $(V_j, \nabla_j)=C_{\tilde U}^{[-1]} (\cE_{j}, \theta _{j})$
	(see \cite[5.1]{La-Bog-normal}) and  $(\cE_{j+1}, \theta _{j+1})= (\Gr _{S_j} (V_j, \nabla_j))^{**}$ for Simpson's filtration $S^{\bullet}_j$ on $(V_j, \nabla_j)$  (see Theorem \ref{existence-of-gr-ss-Griffiths-transverse-filtration}).  
	By \cite[Lemma 2.4 and Theorem 0.4]{La-Bog-normal} we have for $i=1,...,n-1$
	$$\int_X \Delta (V_{m-1}) L_1...\widehat{L_i}...L_{n-1}\ge  \int_X \Delta (\cE_{m}) L_1...\widehat{L_i}...L_{n-1}\ge 0$$
	for $m\ge 0$. Since by \cite[Corollary 5.11]{La-Bog-normal} we have
	$$\int_X \Delta (V_m) L_1...\widehat{L_i}...L_{n-1}=p^2  \int_X \Delta (\cE_{m}) L_1...\widehat{L_i}...L_{n-1},$$
	we see that 
	$$ \int_X \Delta (\cE_{m}) L_1...\widehat{L_i}...L_{n-1}=0$$
	for all $m\ge 0$. 
	
	Now let us write $p^m=q_mr+r_m$ for some integers $q_m$ and $0\le r_m<r$ and
	let $M:=(M_1,..., M_{n-1})$ be a collection of nef line bundles on $X$. To show (1) we use arguments similar to that from the  proof of \cite[Lemma 3.10]{La-JEMS}.
	More precisely, assume that $(W,\nabla_W)\subset (V,\nabla)$ is $M$-destabilizing. Then we 
	set $(\cF_0,\eta _0):=(\Gr_{S_{-1}}(W,\nabla_W))^{**}$ and we
	consider the sequence 
	$$(\cF_m,\eta_m):= (\Gr _{S_{m-1}} C^{[-1]}(\cF_{m-1}, \eta_{m-1}))^{**}.$$
	Then  $\cF_m (-q_mc_1(\cF_0))$ give  subsheaves of $\cE_m':=\cE_m (-q_mc_1(\cE_0))$, whose slopes grow to infinity. But Theorem \ref{boundedness} implies that the set $\{  \cE_m'\}_{m\ge 0}$ is bounded, a contradiction. Similar argument works for $(\cE, \theta)$, showing (1). This and \cite[Lemma 2.4 and Theorem 0.4]{La-Bog-normal} imply that if 
	$(M_1,..., M_{n-2})$ is a collection of nef line bundles on $X$ then 
	$$\int_X \Delta (V) M_1...M_{n-2}\ge  \int_X \Delta (\cE^{**}) M_1...M_{n-2}\ge 0.$$
	By Serre's theorem, we can find some positive integer $a$  such that all line bundles $M_i'=M_i^{-1}\otimes L_i^{\otimes a}$ are globally generated, so also nef. So the above inequalities applied to collections 	$(M_1', M_2,..., M_{n-2})$, $(L_1,M_2',M_3,..., M_{n-2})$,..., give
	$$\int_X \Delta (V) M_1...M_{n-2}\le a\int_X \Delta (V) L_1M_2...M_{n-2}\le ...\le a^{n-2} \int_X \Delta (V) L_1...L_{n-2}=0.$$
	This implies (2).
	
	\medskip
	
	To prove (3) it is sufficient to prove that all $\cE_m|_U$ and $V_m|_U$ are locally free and the maps
	$$ \quad\Gr_{S_{m-1}} V_{m-1} \to \cE_m$$
	are isomorphisms on $U$ for $m\ge 0$. The proof is by induction on the dimension $n$ of $U$.
	For $n=2$ the assertion follows from the fact that reflexive sheaves on smooth surfaces are locally free and from Lemma \ref{Chern-classes-filtrations} (applied after restricting to $U$).

	Let us assume that $n\ge 3$ and the above assertions hold for varieties of dimension $<n$.
For this fixed dimension, we also argue by induction on the rank, with rank $1$ being covered by Case 1.	
	 Taking some field extension  $k\subset K$ if necessary, we can assume that $k$ is uncountable.
	By assumption there exists a closed subset $Z\subset X$ of codimension $\ge 3$ 
	such that every point $x\in X\backslash Z$ has 
an open neighbourhood $U_x\subset X\backslash Z$
for which there exists a lifting $\tilde U_x$ of $U_x$ and a lifting $F_{\tilde U_x}: \tilde U_x\to \tilde U_x$ of the Frobenius morphism $F_{U_x}: U_x\to U_x$. Moreover, $(U_x\cap U, \cO_{\tilde U_x}|_{U_x\cap U})$
is isomorphic to $(U_x\cap U, \cO_{\tilde U}|_{U_x\cap U})$. 	
Since $U= X_{reg}$, every point of $U$ has an open neighbourhood $U_x\subset U$ for  which there exists a lifting 
	$F_{\tilde U_x}: \tilde U_x\to \tilde U_x$ of $F_{U_x}$, where $\tilde U_x=(U_x, 
	\cO_{\tilde U}|_{U_x})$. So we can assume that $Z\cap U=\emptyset$.
	Let us set $Y:=X\backslash U$ and $T_m:=\cE_m/(\Gr_{S_{m-1}} V_{m-1})$ for $m\ge 0$.

	\medskip
	
	\begin{Lemma}\label{loc-free-outside-finite-set}
		For every $m\ge 0$, $U\cap \Supp T_m$ consists of a finite number of points.	
	\end{Lemma}
	
	\begin{proof} Without loss of generality we can assume that $L_1$ is very ample.
		By Bertini's theorem, Corollary \ref{hyperplane-section-of-reflexive} and \cite[Lemma 1.1.12]{HL}, for a very general hypersurface $H\in  |L_1|$ with closed embedding  $\imath: H\hookrightarrow X$  the following conditions are satisfied:
		\begin{enumerate}
			\item $H$ is irreducible and normal, 
			\item $H_{reg}= H\cap X_{reg}$ and hence $U_H:=U\cap H\subset H$ is big,
			\item $U_H$ has a lifting $\tilde U_H$ to $W_2(k)$ compatible with $\tilde U$ (this follows from \cite[proof of Theorem 11]{La-Inv}),
			\item  $\dim Z\cap H<\dim Z$ and hence $Z\cap H$ has codimension $\ge 3$ in $H$,
			\item $H$ has $F$-liftable singularities  with respect to $\tilde U_H$,
			\item for every irreducible component $S$ of $\Supp T_m$, $m\ge 0$, if $S\cap U\ne \emptyset$ and $\dim S>0$ then $S\cap U\cap H\ne \emptyset$,
			\item  $\imath ^*\cE _m$ for $m\ge 0$ and $\imath ^*V _m$ for $m\ge -1$  are reflexive,
			\item  $\imath^*(\Gr _{S_m}V_m)$ for $m\ge -1$ are torsion free.
		\end{enumerate} 
		The above conditions and functoriality of the inverse Cartier transform imply that
		$$\imath^*(V _m, \nabla_m)=C^{[-1]}_{\tilde U_H}\left(\imath^*(\cE _m, \theta_m)\right)$$
		for all $m\ge 0$. Moreover, by \cite[Theorem 4.16]{La-Bog-normal},  all restrictions $\imath^*(\cE _m, \theta _{m})$ are  slope $(\imath^*L_2,...,\imath^*L_{n-1})$-semistable. It follows that 
		$${  \xymatrix{
				\imath^*(V_{-1}, \nabla_{-1}) \ar[rd]^{(\Gr _{\imath^*S_{-1}})^{**}}	&& \imath^*(V_0, \nabla _0)\ar[rd]^{(\Gr _{\imath^*S_0})^{**}}&& \imath^*(V_1, \nabla _1)\ar[rd]^{(\Gr _{\imath^*S_1})^{**}}&\\
				&	\imath^*(\cE_0, \theta _0)\ar[ru]^{C_{\tilde U_H}^{[-1]}}&&\imath^*	(\cE_1, \theta_1)\ar[ru]^{C_{\tilde U_H}^{[-1]}}&&...\\
		}}$$
		is  a {Higgs--de Rham sequence}. By the induction assumption all $\imath^*\cE_m|_{U_H}$ and $\imath^*V_m|_{U_H}$ are locally free and the maps
		$$\quad\imath^*(\Gr_{S_{m-1}} V_{m-1}) \to \imath^*\cE_m $$
		are isomorphisms on $U_H$ for $m\ge 0$. So $\imath^*(T_m|_{U})=(\imath^*T_m)|_{U_H}=0$ for all $m\ge 0$.
		By (6) for every $m\ge 0$, $U\cap \Supp T_m$ consists of a finite number of points.	
	\end{proof}
	
	\medskip
	Now as in the proof of Theorem \ref{reflexive-Delta-are-loc-free}  we consider three cases depending on stability properties of the Higgs sheaves $(\cE _m, \theta_m)$.

	\medskip
	
	\emph{Case 1.} Assume that all $(\cE _m, \theta_m)$ are slope $L$-stable.
	
	Let us fix a $k$-point $x\in U$. Since $Z\subset Y$, $x\not\in Y$ and $x\not\in Z$.
	By Theorem \ref{Bertini's-theorem}, taking $s\gg 0$  we can find a hypersurface $H\in  |L_1^{\otimes s}|$ such that if $\imath: H\hookrightarrow X$ denotes the  closed embedding then the following conditions are satisfied:
	\begin{enumerate}
		\item $x\in H$, 
		\item $H$ is irreducible and normal, 
		\item $H_{reg}= H\cap X_{reg}$ and $U_H:=H\cap U\subset H$ is big,
		\item $H\cap U$ has a lifting $\tilde U_H$ to $W_2(k)$ compatible with $\tilde U$ (this follows from \cite[proof of Theorem 11]{La-Inv}),
		\item  $\dim Z\cap H<\dim Z$ and hence $Z\cap H$ has codimension $\ge 3$ in $H$,
		\item $H$ has $F$-liftable singularities  with respect to $\tilde U_H$,
		\item for every irreducible component $S$ of $\Supp T_m$, $m\ge 0$, we have either
		$\dim (S\cap H)<\dim S$ or $S\cap H\subset \{ x\}$ (here we use Lemma \ref{loc-free-outside-finite-set}),
		\item   $\imath ^*\cE _m$ for $m\ge 0$ and $\imath ^*V _m$ for $m\ge -1$  are reflexive 
		on $H\backslash \{x\}$.
	\end{enumerate} 
	
	Let us recall that  by Corollary \ref{hyperplane-section-of-reflexive} all $\imath ^*\cE _m$ and $\imath ^*V _m$ are torsion free. 
	Let $(\tilde V_m, \tilde\nabla  _m):=\imath^{[*]}(V _m, \nabla_m)$ and
	$(\cF _m, \tilde\theta_{m}):=\imath^{[*]}(\cE _m, \theta_m)$. 
	Note that  $(\tilde V_m, \tilde\nabla  _m)=C^{[-1]}_{\tilde U_H}(\cF _m, \tilde\theta_{m})$
	for $m\ge 0$.
	By \cite[Theorem 5.4]{La-Bog-normal}, all $(\cF _m, \tilde\theta _{m})$ 
	are  slope $(\imath^*L_2,...,\imath^*L_{n-1})$-stable. This implies that all $(\tilde V_m, \tilde\nabla  _m)$
	are  slope $(\imath^*L_2,...,\imath^*L_{n-1})$-stable. By \cite[Theorem 0.4]{La-Bog-normal} and Lemma \ref{inequality-on-Delta}
	we have
	$$0\le \int_H \Delta (\cF _m) \imath^*L_2...\widehat{\imath^*L_i}...\imath^*L_{n-1}\le s \int_X \Delta (\cE _m) L_1...\widehat{L_i}...L_{n-1}=0$$
	for $i=2,...,n-1$. 
	So by the induction assumption all $\cF_m|_{H_{reg}}$ are locally free. Similarly, all  $\tilde V_m|_{H_{reg}}$
	are locally free.
	
	\medskip
	
	Let us define a filtration $\tilde S^{\bullet}_m$ of $\tilde V_m$ by taking as $\tilde S^{i}_m$ the saturation of $\imath^* S_m^i\subset \tilde V_m$. Let us also set $(\cF_m, \tilde \theta _m):= (\Gr _{\tilde S_{m-1}}\tilde V_{m-1})^{**}$.
	Then 
	we have a Higgs-de Rham sequence over $H$:
	$${  \xymatrix{
			(\tilde V_{-1}, \tilde\nabla  _{-1})\ar[rd]^{(\Gr _{\tilde S_{-1}})^{**}}&	& (\tilde V_0, \tilde\nabla  _0)\ar[rd]^{(\Gr _{\tilde S_0})^{**}}&& (\tilde V_1, \tilde\nabla _1)\ar[rd]^{(\Gr _{\tilde S_1})^{**}}&\\
			&		(\cF_0, \tilde \theta _0)\ar[ru]^{C_{\tilde U_H}^{[-1]}}&&	(\cF_1, \tilde\theta_1)\ar[ru]^{C_{\tilde U_H}^{[-1]}}&&...\\
	}}$$

	Let
	$(\cG_m, \theta_{\cG_m})$ be the cokernel of $\imath^*(\cE _m, \theta_m)\to (\cF _m, \tilde\theta _{m})$.
	Note that all $\cG_m$ are supported on $\{ x\}$.
	By construction and functoriality of the inverse Cartier transform, there exists an open subset $V\subset H_{reg}$ containing $x$ for which there exists a lifting $F_{\tilde{V}}: \tilde V\to \tilde V$ of $F_V$ and such that
	$$\imath^*(V _m, \nabla_m)|_V=C^{-1}_{\tilde V}\left(\imath^*(\cE _m, \theta_m)|_V\right)\quad\hbox{and}
	\quad 
	(\tilde V_m, \tilde\nabla  _m)|_V=C^{-1}_{\tilde V}\left(
	\left(\cF _m, \tilde\theta_{m}\right)|_V
	\right).$$
	Therefore 
	$$ (\tilde V_m, \tilde\nabla  _m) /\imath^*(V _m, \nabla_m) = C^{-1}_{\tilde V}\left(\cG_m, \theta_{\cG_m}\right).$$
	In particular, we get
	$$h^0\left( \tilde V_m/\imath^*V _m  \right)= h^0\left( F_H^*\cG_m \right)=p^{n-1}h^0(\cG_m).$$
	Applying Lemma \ref{passing-to-grading-for-hypersurfaces} on  $X_{reg}$, we see that
	$$h^0(\cG_{m})=h^0\left(\cF_{m}/ \imath^*\cE_{m} \right)=h^0\left(\tilde V_{m-1}/\imath^*V _{m-1}  \right)
	$$
	for $m\ge 0$. Let us set $\cG:= \tilde V_{-1}/\imath^*V _{-1}$. It follows that
	$$\chi (X, \imath_*\cG_{m})=h^0(H,\cG_m)=p^{(n-1)m}h^0(H, \cG)$$
	for all $m\ge 0$.
	Now let us set $\cE_m':=\cE_m (-q_mc_1(\cE_0))$ and $\cF_m':=\cF_m (-q_mc_1(\cF_0))$, where $p^m=q_mr+r_m$ 
	for some integers $q_m$ and $0\le r_m<r$. Note that $\cG_m$ is isomorphic to the cokernel of 
	$\imath^*\cE_m'\to \cF_m'$ so we have
	\begin{align*}
		p^{(n-1)m} h^0 (X,  \imath_*\cG )= \chi (X, \imath_*\cG_{m})=
		\chi (H, \cF'_m ) -\chi (X, \cE'_m)
		+\chi (X, \cE '_m (-H))\\
	\end{align*}
	But the sets $\{  \cE '_m\}_{m\ge 0}$, $\{ \cE '_m(-H)\}_{m\ge 0}$ and 
	$\{ \cF '_m \}$ are bounded by Theorem \ref{boundedness}
	(in these sets $c_1(\cdot )$ takes only finitely many values and $\int_X \Delta (\cdot ) L_1...\widehat{L_i}...L_{n-1}= 0$). So the right hand side of the above equalities is bounded independently of $m$. Therefore  we have $ h^0 (X,  \imath_* \cG )=0$ and  $\cG =0$. This implies that $V$ is locally free at $x$ and $V|_U$ is locally free on $U$. 
	Applying this fact to truncated Higgs--de Rham flows, we see that also  $\cE_m|_U$ and $V_m|_U$ are all locally free for all $m\ge 0$ (alternatively, we can use the fact that $\cG_m=0$ to conclude that $\cE_m|_U$ is locally free; then $V_m|_U$ is locally free by construction). The fact that for $m\ge 0$ the maps
	$$\Gr_{S_{m-1}} V_{m-1} \to \cE_m $$
	are isomorphisms on $U$ follows from Lemma \ref{Chern-classes-filtrations}.

	\medskip
	
	\emph{Case 2.} Assume that $(V, \nabla)$ is not slope $L$-stable.
	Let $0= (V_0, \nabla_0)\subset (V_1, \nabla_1)\subset ...\subset (V_s, \nabla_s)=(V, \nabla)$ be a Jordan--H\"older filtration of $(V, \nabla)$ and let us set 
	$$(\cF_j, \tilde \nabla_j):=
	((V_j, \nabla_j)/(V _{j-1} , \nabla_{j-1}))^{**}.$$
	The same arguments as that in Case 2 of the proof of Theorem \ref{reflexive-Delta-are-loc-free}
	show that  $$\int _X \Delta (\cF _j) L_1...\widehat{L_i}...L_{n-1} = 0$$
	for $i=1,...,n-1$. The only difference is that one needs to use  \cite[Theorem 0.4]{La-Bog-normal} instead of  \cite[Theorem 3.4]{La-Bog-normal}. By induction on the rank, all $\cF_j|_U$ are locally free. Then Lemma \ref{Chern-classes-filtrations} shows that
	$V |_{U}$ and $(V_j/V_{j-1})|_{U}$ are locally free. This also implies the remaining assertions.
	
	\medskip
	
	\emph{Case 3.} If Case~1 does not apply, choose $m\ge0$ such that
	$(\cE_m,\theta_m)$ is not slope $L$-stable.
	Applying Case~2 to $(\cE_m,\theta_m)$, we see that $\cE_m|_U$ is locally free.
	 Then  Lemmas \ref{loc-free-passing-to-graded} and \ref{Chern-classes-filtrations} show that $\Gr_{S_{m-1}} V_{m-1} |_U\to \cE_m|_U$ is an isomorphism 
	and hence $V_{m-1}|_U$ is locally free. If $m=0$, this finishes the descent.
	Otherwise, by construction every point $x\in U$
	has an open neighbourhood  $U_x\subset U$ containing $x$ with a lifting $F_{\tilde{U_x}}: \tilde U_x\to \tilde U_x$ of $F_{U_x}$ and such that
	$$(V _{m-1}, \nabla_{m-1})|_{U_x}=C^{-1}_{\tilde U_x}\left(\cE _{m-1}, \theta_{m-1}\right)|_{U_x},$$
	which by flatness of $F_{U_x}$ implies that $\cE _{m-1}|_{U_x}$ is locally free. So the required assertions follow by the descending induction.
\end{proof}

\medskip

\begin{Remark}\label{proof-ofmain-Higgs}
By Theorem \ref{existence-of-gr-ss-Griffiths-transverse-filtration} every $(V, \nabla)$ as in Theorem \ref{reflexive-loc-free-Delta-Higgs}  admits a filtration $N^{\bullet}$ satisfying assumptions of this theorem. 
This implies Theorem \ref{main-Higgs}.
\end{Remark}

\medskip

\begin{Corollary} 
	Let $X,U, L$ be as in Theorem \ref{reflexive-loc-free-Delta-Higgs}.
	Let  $(V, \nabla)$ be a reflexive $\cO_X$-module with an integrable $\lambda$-connection for some $\lambda\in k$. Assume that $(V, \nabla)$ is  slope $L$-semistable of rank
	$r\le p$ with $$\int_X \ch_1 (V) L_1...L_{n-1}=0$$ and
	$$\int_X \ch_2 (V) L_1...\widehat{L_i}...L_{n-1}\ge 0$$
	for $i=1,...,n-1$. Then the following conditions are satisfied:
	\begin{enumerate}
		\item $[c_1(V)]=0\in \bar B^1(X)$ and $\int_X \ch_2 (V) =0$,
		\item if $M=(M_1,..., M_{n-1})$ is a collection of nef line bundles on $X$ then $(V, \nabla)$ is slope $M$-semistable,
		\item $V|_{X_{reg}}$ is locally free.
	\end{enumerate}
	Moreover, if $N^{\bullet }$ is a Griffiths transverse filtration on $(V, \nabla)$  such that  $(\cE:=\Gr _{N} V, \theta )$ is  slope $L$-semistable then the following conditions are satisfied:
	\begin{enumerate}
		\item $[c_1(\cE)]=0\in \bar B^1(X)$ and $\int_X \ch_2 (\cE^{**})= 0$,
		\item if $M=(M_1,..., M_{n-1})$ is a collection of nef line bundles on $X$ then $(\cE , \theta)$ is slope $M$-semistable,
		\item $\cE|_{X_{reg}}$ is locally free.
	\end{enumerate}
\end{Corollary}

\begin{proof}
	By \cite[Lemma 2.5 and Theorem 0.4]{La-Bog-normal} we have
	\begin{align*}
		0\le \int_X \Delta (V) L_1...\widehat{L_i}...L_{n-1}=
		\int_X c_1 (V) ^2L_1...\widehat{L_i}...L_{n-1}-	
		2r	\int_X \ch_2 (V) L_1...\widehat{L_i}...L_{n-1}\\
		\le \frac{(\int_X c_1 (V) L_1...L_{n-1})^2}{L_1...{L_i^2}...L_{n-1}}=0
	\end{align*}
	So by Theorem \ref{reflexive-loc-free-Delta-Higgs},  $V$ satisfies (2) and (3). Moreover, we have
	$$\int_X c_1 (V) ^2L_1...\widehat{L_i}...L_{n-1}=\int_X \ch_2 (V) L_1...\widehat{L_i}...L_{n-1}= 0$$
	for $i=1,...,n-1$. Now $[c_1(V)]=0\in \bar B^1(X)$ by \cite[Corollary 3.17]{La-Bog-normal}. So by Theorem \ref{reflexive-loc-free-Delta-Higgs},  $V$ satisfies also (1). The assertions for $\cE$ follow easily from this and from  Theorem \ref{reflexive-loc-free-Delta-Higgs}.
\end{proof}

\section{Local freeness for modules with logarithmic $\lambda$-connection }

 The main aim of this section is to provide an analogue of Theorem \ref{reflexive-loc-free-Delta-Higgs} for logarithmic $\lambda$-connections.  The proof is divided into two cases: outside of the fixed divisor and on the divisor. In the first case, the proof is essentially the same as that of Theorem \ref{reflexive-loc-free-Delta-Higgs}. In the second case, the proof follows the idea of proof of \cite[Theorem 2.2]{La-JEMS}.
 
\subsection{The main result and its corollaries} 
 
In this section we fix the following notation. $k$ is an algebraically closed field of characteristic $p>0$.  $X$ is a normal projective $k$-variety of dimension $n $ with a reduced divisor $D\subset X$.
$U$ denotes the log smooth locus of $(X,D)$ and we assume that the pair $(U, D_U=D\cap U)$ has a lifting $(\tilde U, \tilde D_U)$ to $W_2(k)$. We also assume that for some closed  subset $Z\subset X\backslash U$ of codimension $\ge 3$ in $X$, the pair $(X\backslash Z,D\backslash Z)$ is locally $F$-liftable with respect to $(\tilde U, \tilde D_U)$.

$(V, \nabla)$ is a reflexive $\cO_X$-module of rank  $r\le p$ with an integrable logarithmic  $\lambda$-connection for some $\lambda\in k$. We fix a collection $L:=(L_1,...,L_{n-1})$ of ample line bundles on $X$ and we assume that  $(V, \nabla)$ is  slope $L$-semistable  and 
$$\int_X \Delta (V) L_1...\widehat{L_i}...L_{n-1}= 0$$
for $i=1,...,n-1$. We also fix a Griffiths transverse filtration $N^{\bullet }$  on $(V, \nabla)$ such that the associated graded Higgs sheaf  $(\cE:=\Gr _{N} V, \theta )$ is  slope $L$-semistable.
The following theorem follows easily from Propositions \ref{local-freeness-log-case-outside-D} and  \ref{local-freeness-log-case-on-D} by induction on the dimension.
 
 \begin{Theorem}\label{local-freeness-log-Higgs-case}
 	Assume that every irreducible component $Y$ of every nonempty
 	intersection of irreducible components of $D$, endowed with its
 	reduced induced structure, is normal and is a local complete
 	intersection in $X$. Assume moreover that $Y\cap U$ is big in $Y$
 	and that $\operatorname{codim}_Y(Z\cap Y)\ge 3$ for every such $Y$.
Then the following conditions are satisfied:
\begin{enumerate}
		\item if $M=(M_1,..., M_{n-1})$ is a collection of nef line bundles on $X$ then 
	both $(V, \nabla)$ and $(\cE , \theta)$ are slope $M$-semistable,
	\item $\int_X \Delta (V)=\int _{X}\Delta (\cE^{**}) = 0$,	
	\item $V|_{U}$ and $\cE |_U$ are locally free.
\end{enumerate}
\end{Theorem}

  \begin{Corollary}\label{strong-log-freeness}
  	Let  $(X,D)$ be as in Theorem \ref{local-freeness-log-Higgs-case} and
  	let $\cF$ be a coherent  reflexive  $ \cO_X$-module with 	
  	$$\int_X \Delta (\cF) L_1...\widehat{L_i}...L_{n-1}= 0$$
  	for $i=1,...,n-1$.
  	Assume that $\cF$ has a filtration $M_{\bullet}$ by $\cO_X$-submodules such that all quotients are torsion free of rank $\le p$ with $\mu
  	_{L}(\Gr _j^M\cF)=\mu _{L}(\cF)$ for all $j$.  Let us also assume that each factor
  	has a structure of a slope $L$-semistable $\cO_X$-module with an integrable logarithmic  $\lambda$-connection on $(X, D)$.
  	Then $\cF|_U$ is locally free and $\int_X \Delta (\cF)=0$. Moreover, every $\Gr _j^M\cF|_U$ is locally free and we have
  	$\int_X \Delta ((\Gr _j^M\cF)^{**})= 0$.
  \end{Corollary}
  
  \begin{proof}
  	Let $r:=\rk \cF $, $\cF_j:= (\Gr _j^M\cF)^{**}$ and $r_j:= \rk \cF_j$. By \cite[Lemma 2.5 and Theorem 0.4]{La-Bog-normal} we have
  	\begin{align*}
  		0=	\frac{\int _X \Delta  (\cF)L_1...\widehat{L_i}...L_{n-1}}{r}\ge \sum _j \frac{\int _X \Delta  (\cF_j)L_1...\widehat{L_i}...L_{n-1}}{r_j}\ge 0.
  	\end{align*}
  	So 	$$\int_X \Delta (\cF_j) L_1...\widehat{L_i}...L_{n-1}= 0$$
  	for $i=1,...,n-1$. Now  the required assertions follow from Theorem \ref{local-freeness-log-Higgs-case} and
  	 Lemma \ref{Chern-classes-filtrations}.
  \end{proof}
  
  \begin{Remark}
  	The proof of the above corollary gives a simple replacement to \cite[2.1]{La-JEMS}. In particular, Theorem \ref{local-freeness-log-Higgs-case} for varieties of dimension $\le n$ implies Corollary \ref{strong-log-freeness} for varieties of dimension $\le n$.
  \end{Remark}
  
\begin{Corollary} \label{existence-of-Higgs-de_Rham-flow}
	Let  $(X,D), U$ etc. be as in Theorem \ref{local-freeness-log-Higgs-case}.
Let $U'\subset U$ be a big open subset and let $\imath: U'\hookrightarrow X$ be	
	the open embedding.
Let $(V', \nabla')$ be a vector bundle with an integrable logarithmic $\lambda$-connection on $(U',D_{U'})$ 
such that $(V,\nabla)=\imath_*(V', \nabla')$ is   slope $L$-semistable  with
$$\int_X \Delta (V) L_1...\widehat{L_i}...L_{n-1}= 0$$
for $i=1,...,n-1$. Then on $U'$ there exists a canonical Higgs--de Rham flow 
	$${  \xymatrix{
		(V_{-1}, \nabla _{-1}) \ar[rd]^{\Gr _{F_{-1}}}	&& (V_0, \nabla _0)\ar[rd]^{\Gr _{F_0}}&& (V_1, \nabla _1)\ar[rd]^{\Gr _{F_1}}&\\
		&	(\cE_0, \theta _0)\ar[ru]^{C_{\tilde X}^{-1}}&&	(\cE_1, \theta_1)\ar[ru]^{C_{\tilde X}^{-1}}&&...\\
}}$$
beginning with $(V_{-1}, \nabla _{-1})=(V', \nabla')$ and such that
\begin{enumerate}
	\item  $\int_X \Delta (\imath_*V_j)=\int _{X}\Delta (\imath_*\cE_j) = 0$,
	\item $\imath_*  (V_j, \nabla _j)$ and $\imath_* (\cE_j, \theta_j)$ are slope $M$-semistable for any collection $M$ of nef line bundles on $X$. 
\end{enumerate}
\end{Corollary}

 \subsection{Local freeness outside of $D$}

The following proposition shows local freeness in Theorem \ref{local-freeness-log-Higgs-case} outside of the divisor $D$. It is an analogue of  \cite[Lemma 3.9]{La-JEMS} for singular varieties. 

\begin{Proposition} \label{local-freeness-log-case-outside-D}
The following conditions are satisfied:
	\begin{enumerate}
			\item if $M=(M_1,..., M_{n-1})$ is a collection of nef line bundles on $X$ then 
		both $(V, \nabla)$ and $(\cE , \theta)$ are slope $M$-semistable,
		\item $\int_X \Delta (V)=\int _{X}\Delta (\cE^{**}) = 0$,
		\item $V|_{U\backslash D}$ and $\cE |_{U\backslash D}$ are locally free.
	\end{enumerate}
\end{Proposition}

\begin{proof}
The proof is by induction on the dimension of $X$ and it is analogous to the proof of Theorem \ref{reflexive-loc-free-Delta-Higgs}, so we just sketch the main idea. The proof of (1) and (2) is the same as in that theorem. Now as in Lemma \ref{loc-free-outside-finite-set} we  show that the intersection of $U\backslash D$ with the support of $\cE ^{**}/\cE $ is a finite set of points. 
This is done by restricting to a very general  $H\in |L_1|$ and using the induction assumption on $H$. Then we consider three cases depending on the stability properties of the Higgs--de Rham sequence of $(V, \nabla)$. In the first case, we assume that all systems of logarithmic Hodge sheaves appearing in the Higgs--de Rham sequence of $(V, \nabla)$ are slope stable. Then we restrict to a very general $H\in |L_1|$ containing a fixed $k$-point $x\in U\backslash D$. Using the same arguments as in Case 1 of the proof of Theorem \ref{reflexive-loc-free-Delta-Higgs}, one shows that $V, \cE$ and all sheaves in the Higgs--de Rham sequence  are locally free at $x$.  The remaining cases are done in the same way as in the proof of Theorem \ref{reflexive-loc-free-Delta-Higgs}. 
\end{proof}

\subsection{Local freeness on $D$}

The following proposition shows local freeness in Theorem \ref{local-freeness-log-Higgs-case} at the points of  $D$. It is an analogue of  \cite[Lemma 3.10]{La-JEMS} for singular varieties. 
Note that the assumptions are stronger than that in the proof of Proposition \ref{local-freeness-log-case-outside-D}.

Let $Y$ be an irreducible component of $D$ and let $\imath :Y\hookrightarrow X$ be the corresponding embedding.
We denote by $D_Y$ the divisor $Y\cap \overline{D\backslash Y}$ on $Y$.
In the proposition below, $\imath ^{[*]}(V, \nabla)$ denotes a $\imath ^{[*]}T_{\lambda}(X,D)$-module and 
$\imath ^{[*]}(\cE , \theta)$ a $\Sym ^{\bullet }\imath ^{[*]}T_{X}(\log D)$-module.

\begin{Proposition} \label{local-freeness-log-case-on-D}
	Assume that $Y$ is a normal Cartier divisor, $U_Y:=Y\cap U$
	is big in $Y$, and $Z_Y:=Z\cap Y$ has codimension $\ge 3$ in $Y$.
	Assume moreover that every irreducible component $T$ of every
	nonempty intersection of irreducible components of $D_Y$,
	endowed with its reduced induced structure, is normal and is
	a local complete intersection in $Y$, that $T\cap U_Y$ is big
	in $T$, and that $\operatorname{codim}_T(T\cap Z_Y)\ge 3$.
Let $M=(M_1,..., M_{n-2})$ be any collection of nef line bundles on $Y$.
If Theorem \ref{local-freeness-log-Higgs-case} holds for varieties of dimension $<n$ then the following conditions are satisfied:
	\begin{enumerate}
		\item  $\imath ^{[*]}(\cE , \theta)$ is slope $M$-semistable as a $\Sym ^{\bullet }\imath ^{[*]}T_{X}(\log D)$-module.
	\item Assume in addition that $\Res_Y\nabla$ is nilpotent.
	Let $W_{\bullet}$ be the saturated extension to
	$\imath^{[*]}V$ of its generic monodromy filtration.
	Then the sheaves $(\Gr_i^W\imath^{[*]}V)^{**}$ carry canonical
	integrable logarithmic $\lambda$-connections on $(Y,D_Y)$.
	They are slope $M$-semistable of slope
	$\mu_M(\imath^{[*]}V)$.
	
	\item Let $M_{\bullet}$ be the saturated extension to
	$\imath^{[*]}\cE$ of the generic monodromy filtration of
	$\Res_Y\theta$.
	Then the sheaves $(\Gr_i^M\imath^{[*]}\cE)^{**}$ carry canonical
	logarithmic Higgs fields on $(Y,D_Y)$.
	They are slope $M$-semistable of slope
	$\mu_M(\imath^{[*]}\cE)$.	
		\item $\int_Y \Delta (\imath ^{[*]}V)=\int_Y \Delta (\imath ^{[*]}\cE)= 0$.
		\item $V$ and $\cE$ are locally free along $Y\cap U$.
	\end{enumerate}
\end{Proposition}
\begin{proof}
	Our assumptions imply that $(Y,D_Y)$ is almost liftable and
	locally $F$-liftable outside of $Z\cap Y$.
	Let us set $(V_{-1},\nabla_{-1}):=(V,\nabla)$ and
	$S_{-1}^{\bullet}:=N^{\bullet}$. Let
	$${\xymatrix{
			(V_{-1},\nabla_{-1})
			\ar[rd]^{(\Gr_{S_{-1}})^{**}}
			&& (V_0,\nabla_0)\ar[rd]^{(\Gr_{S_0})^{**}}
			&& (V_1,\nabla_1)\ar[rd]^{(\Gr_{S_1})^{**}}&\\
			&(\cE_0,\theta_0)\ar[ru]^{C_{\tilde U}^{[-1]}}
			&&(\cE_1,\theta_1)\ar[ru]^{C_{\tilde U}^{[-1]}}
			&&\cdots
	}}$$
	be a Higgs--de Rham sequence.  Set
	$
	T_m:=\cE_m/(\Gr_{S_{m-1}}V_{m-1}).
	$
	By reflexivity, $\cE_m|_U$ and $V_m|_U$ are locally free
	outside closed subsets of codimension $\ge3$ in $U$.
	Moreover, Proposition \ref{local-freeness-log-case-outside-D}
	and Lemma \ref{Chern-classes-filtrations} imply that
	$U\cap\Supp T_m$ has codimension $\ge3$ in $U$.

	We try to restrict the above Higgs de--Rham sequence to $Y$ and we consider  the reflexive hulls $\tilde V_m=\imath^{[*]}V _m$ and
$(\cF _m, \tilde\theta_{m})=\imath^{[*]}(\cE _m, \theta_m)$.  Note that every $(\cF _m, \tilde\theta_{m})$ is a $\Sym ^{\bullet }\imath ^{[*]}T_{X}(\log D)$-module but the inverse  Cartier transform is not well  defined for $\Sym ^{\bullet }\imath ^{[*]}T_{X}(\log D)$-modules. However, we can use  \cite[Lemma 3.7]{La-JEMS} and the inverse Cartier
	transform on the total space of the normal bundle of
	$U\cap Y$. Then we can apply the boundedness argument from
	\cite[Lemma 3.10]{La-JEMS}.
	Indeed, the suitably twisted reflexive sheaves $\cF_m$
	form a bounded family by Theorem \ref{boundedness} and
	Proposition \ref{needed-for-independence-of-compactification}.
	It follows that every $(\cF_m,\tilde\theta_m)$ is slope
	$M$-semistable, as a
	$\Sym^\bullet\imath^{[*]}T_X(\log D)$-module, for every
	collection $M$ of nef line bundles on $Y$.
	
	If $n=2$, reflexivity gives local freeness of $V$ and
	$\cE_0$ on $U$, and
	Lemma \ref{Chern-classes-filtrations} gives
	$\cE|_U=\cE_0|_U$.
	The semistability and monodromy arguments below then give
	the remaining assertions on the smooth curve $Y$.
	We therefore assume that $n\ge3$.
	Since $Y$ is a Cartier divisor and $U\cap Y$ is big in $Y$,
	the preceding codimension estimates imply that
	$\imath^{[*]}\cE\simeq\cF_0$.
	Thus the semistability just proved gives (1).
	
	Since $(\cE_0,\theta_0)$ is the reflexive hull of a system
	of Hodge sheaves, $\theta_0$ is nilpotent.
	Let $M^0_{\bullet}$ be the saturated extension of the
	generic monodromy filtration of $\Res_Y\theta_0$ and set
	$$
	(\cF'_0,\theta'_0):=
	\bigl(\Gr_{M^0}(\cF_0,\tilde\theta_0)\bigr)^{**}.
	$$
	By the arguments of \cite[Lemmas 3.2 and 3.3]{La-JEMS},
	applied on a big open subset and followed by reflexive
	extension, this is a logarithmic Higgs sheaf on $(Y,D_Y)$.
	It is slope $M$-semistable for every collection $M$ of
	nef line bundles on $Y$, and its graded summands have
	the same slope.
	
	Let $N_Y^\bullet$ be the saturated filtration on
	$\tilde V_{-1}$ induced by $N^\bullet$.
	Restriction commutes with taking the associated graded
	on a big open subset of $Y$, so
	$$
	(\Gr_{N_Y}\tilde V_{-1})^{**}\simeq\cF_0.
	$$
	Consequently, by
	\cite[Lemma 2.4 and Theorem 0.4]{La-Bog-normal}, for every
	collection $P_1,\ldots,P_{n-3}$ of nef line bundles on $Y$
	we have
	$$
	0\le
	\int_Y\Delta(\cF'_0)P_1\cdots P_{n-3}
	\le
	\int_Y\Delta(\cF_0)P_1\cdots P_{n-3}
	\le
	\int_Y\Delta(\tilde V_{-1})P_1\cdots P_{n-3}.
	$$
	
	Let $A=(A_1,\ldots,A_{n-3})$ be a collection of ample
	line bundles on $X$ and set $B_i:=\imath^*A_i$.
	By Lemma \ref{inequality-on-Delta} and
	Proposition \ref{local-freeness-log-case-outside-D},
	$$
	\int_Y\Delta(\tilde V_{-1})B_1\cdots B_{n-3}
	\le
	\int_X\Delta(V)\cO_X(Y)A_1\cdots A_{n-3}=0.
	$$
	For the nef line bundles $P_i$ above, choose positive
	integers $a_i$ such that
	$B_i^{\otimes a_i}\otimes P_i^{-1}$ is nef.
	Multilinearity and nonnegativity on products of nef
	line bundles give
	$$
	0\le
	\int_Y\Delta(\tilde V_{-1})P_1\cdots P_{n-3}
	\le
	\left(\prod_{i=1}^{n-3}a_i\right)
	\int_Y\Delta(\tilde V_{-1})B_1\cdots B_{n-3}
	=0.
	$$
	Together with the preceding inequalities, this proves
	$$
	\int_Y\Delta(\tilde V_{-1})
	=
	\int_Y\Delta(\cF_0)
	=
	\int_Y\Delta(\cF'_0)=0,
	$$
	and hence (4).
	
	Apply the induction assumption to
	$(\cF'_0,\theta'_0)$, viewed as a logarithmic
	$0$-connection on $(Y,D_Y)$.
	It follows that $\cF'_0|_{U\cap Y}$ is locally free.
	The local argument of
	Lemma \ref{Chern-classes-filtrations}, applied to
	$M^0_{\bullet}$, shows that $\cF_0|_{U\cap Y}$ and its
	monodromy graded pieces are locally free.
	These graded pieces coincide on $U\cap Y$ with their
	reflexive hulls. This proves (3).
	
	We also have
	$$
	\int_Y\Delta(\cF_0)B_1\cdots B_{n-3}
	=
	\int_X\Delta(\cE_0)\cO_X(Y)A_1\cdots A_{n-3}=0.
	$$
	Equality in Lemma \ref{inequality-on-Delta} therefore
	implies that $Q:=\cF_0/\imath^*\cE_0$ is supported in
	codimension $\ge3$ on $U\cap Y$.
	Since $\cF_0|_{U\cap Y}$ is locally free, the Ext-vanishing
	argument used in the proof of Lemma~1.9 gives
	$$
	\cExt^1_{U\cap Y}
	\bigl((\imath^*\cE_0)|_{U\cap Y},\cO_{U\cap Y}\bigr)=0.
	$$
	Dualizing the restriction sequence
	$$
	0\longrightarrow\cE_0(-Y)\longrightarrow\cE_0
	\longrightarrow\imath_*\imath^*\cE_0\longrightarrow0
	$$
	and applying Nakayama's lemma locally, we obtain
	$
	(\imath^*(\cE_0^*))|_{U\cap Y}
	\simeq(\cF_0^*)|_{U\cap Y}.
	$
	Thus $\cE_0^*$ is locally free along $U\cap Y$,
	and reflexivity implies the same for $\cE_0$.
	
	Finally, $\cE_0/\cE$ is supported in codimension $\ge3$
	on $U$.
	Applying Lemma~1.9 to $N^\bullet$ on a neighbourhood of
	$U\cap Y$ in $U$, we conclude that both $V$ and
	$\cE=\Gr_N V$ are locally free along $U\cap Y$.
	This proves (5).
	
	To prove (2), assume that $\Res_Y\nabla$ is nilpotent.
	The filtration $N_Y^\bullet$ and semistability of
	$(\cF_0,\tilde\theta_0)$ show that every subsheaf of
	$\tilde V_{-1}$ invariant under the restricted
	$\lambda$-connection has slope at most
	$\mu_M(\tilde V_{-1})$.
	The monodromy argument of
	\cite[Lemmas 3.2 and 3.3]{La-JEMS} then gives canonical
	integrable logarithmic $\lambda$-connections on the
	reflexive monodromy graded pieces, all slope
	$M$-semistable of the same slope.
	This proves (2).
\end{proof}

\section{Applications}

In this section we provide a few applications of the obtained results. In particular, we show some applications to schemes defined in zero or mixed characteristic. We also prove some strong restriction theorems for strongly numerically flat bundles and semistable Higgs bundles with vanishing discriminant.

\subsection{Crystalline representations of the \'etale fundamental group}

Let $k$ be an algebraic closure of a finite field of characteristic $p$.
Let $W=W(k)$ be the Witt ring of $k$ and $K$ its field of fractions.
Let $\cX$ be a smooth scheme over $W$ and let
$j:\cX\hookrightarrow\bar\cX$ be an open embedding into a normal
projective scheme over $W$. Assume that the complement of $\cX$
in $\bar\cX$ has codimension $\ge2$.

Set $X:=\cX\otimes_W k$ and $\bar X:=\bar\cX\otimes_W k$.
Assume that $\bar X$ is a normal variety of dimension $n\ge2$
and that $X$ is big in $\bar X$.
We also denote the induced open embedding $X\hookrightarrow\bar X$
by $j$.

We use Higgs--de Rham flows in the sense of Definition~1.10,
allowing a Higgs bundle as the initial $0$-connection.
Preperiodicity after twisting means the existence of such a flow
with a twisted-periodic tail.

As a corollary to Theorem \ref{main-Higgs}, one gets the following
theorem generalizing \cite[Theorem 3.10]{SYZ}.

\begin{Theorem}\label{preperiodic}
	Assume that $\bar X$ has $F$-liftable singularities in codimension $2$
	with respect to the lifting $\cX\otimes_W W_2(k)$ of $X$.
	Let $H$ be an ample line bundle on $\bar X$ and let
	$(\cE,\theta)$ be a rank $r\le p$ Higgs bundle on $X$.
	Then the following assertions hold:
	\begin{enumerate}
		\item If $(\cE,\theta)$ is slope $H$-semistable and
		$$
		\int_{\bar X}\Delta(j_*\cE)H^{n-2}=0,
		$$
		then $(\cE,\theta)$ is preperiodic after twisting.
		
		\item If, in addition,
		$\operatorname{codim}_{\bar X}(\bar X\setminus X)\ge3$,
		then the converse in {\rm (1)} also holds.
		More precisely, every twisted preperiodic Higgs bundle
		$(\cE,\theta)$ on $X$ is slope semistable with respect to
		every ample line bundle on $\bar X$ and satisfies
		$$
		\int_{\bar X}\Delta(j_*\cE)=0.
		$$
	\end{enumerate}
\end{Theorem}

In \cite{SYZ}, the authors defined twisted Fontaine--Faltings
modules and proved their relation to projective representations
of the \'etale fundamental group in the smooth projective case.
Here we only need strict $p$-torsion Fontaine--Faltings modules.

If $\bar\cX$ is smooth over $W$, Theorem \ref{main-Higgs}
allows us to extend the Higgs--de Rham flow to $\bar X$.
Applying the proper correspondence to a twisted-periodic tail
and restricting the resulting representation to $\cX_K$
gives the following corollary
(see \cite[Theorem 6.7]{LSZ} and
\cite[Theorems 0.1, 0.2 and 0.5]{SYZ}).

\begin{Corollary}\label{existence-of-representation}
	Under the assumptions of Theorem \ref{preperiodic}, let
	$(\cE,\theta)$ satisfy the semistability and discriminant
	conditions in {\rm (1)}.
	Assume in addition that $\bar\cX$ is smooth over $W$.
	If $r<p$, then $(\cE,\theta)$ gives rise to a continuous
	projective representation
	$$
	\rho_{\cE}:\pi_1^{\et}(\cX_K)\longrightarrow{\PGL}_r(k).
	$$
	Moreover, if $c_1(j_*\cE)H^{n-1}$ is coprime to $r$, then
	$\rho_{\cE}$ is geometrically absolutely irreducible.
\end{Corollary}

\medskip

\begin{Example} Assume that $p>n+1$ and  $\bar\cX/W$ is smooth.
	Let us consider a canonical Higgs bundle $(\cE=\cE^1\oplus\cE^0, \theta)$ on $X$ defined by $\cE^1:= \Omega _X^{1}$, $\cE^0:=\cO_X$ with $\theta : \cE^1\to \cE^0 \otimes _{\cO_X}\Omega _X$ given by identity. If for some ample line bundle $H$ on $\bar X$ we have
	$$2(n+1)\cdot \int_{\bar X}  c_2(\Omega_{\bar X}^{[1]})H^{n-2}-n\cdot K_{\bar X}^2H^{n-2}=0$$
	and $(\cE,\theta)$ is slope $H$-semistable, then Corollary \ref{existence-of-representation} shows that we have an associated continuous projective representation
	$$\rho_{\cE}: \pi_1^{\et} (\cX_K)\to {\PGL }_{n+1}(k).$$
The finite image of $\rho_{\cE}$ defines a finite \'etale covering of $\cX_K$ on which the associated projective
local system is trivial.
\end{Example}

\begin{Remark}
All the results in this subsection have obvious logarithmic versions that hold because of Theorem
\ref{local-freeness-log-Higgs-case}. 
\end{Remark}

\subsection{Restriction theorems}

As mentioned in the introduction, our results imply some strong restriction theorems. Let us first start with the following  easy result that follows directly from the definition.

\begin{Theorem}\label{restriction-strong-numerical-flatness}
Let $X$ and $Y$  be normally big varieties defined over an algebraically closed field of positive characteristic. Let $\cE$ be a strongly projectively (numerically) flat vector bundle on $X$  and let 
$f: Y\to X$ be any morphism. Then $f^*\cE$  is also  strongly projectively (numerically, respectively) flat. 
\end{Theorem}

\begin{Remark}
Note that in the above theorem we allow $Y$ to be a curve. In this case our assumption implies that $Y$ is projective.  But the assertion becomes more astonishing in higher dimensions, when it implies vanishing of the suitable Chern classes of extensions of  $f^*\cE$ to any small compactification of $Y$.
\end{Remark}

\medskip

We also have an analogous  theorem for semistable modules with a $\lambda$-connection on varieties liftable modulo $p^2$.

\begin{Theorem}\label{restriction-Higgs-bundles}
Let $f: Y\to X$ be a morphism of normally big varieties defined over an algebraically closed field $k$ of characteristic $p>0$.  Assume that there exists a lifting $\tilde f: \tilde Y\to \tilde X$
of $f$ to $W_2(k)$. Assume also that $X$ and $Y$ admit  small compactifications $j: X\hookrightarrow \bar X$ and $j': Y\hookrightarrow \bar Y$ with $F$-liftable singularities in codimension $2$ with respect to $\tilde X$ and $\tilde Y$, respectively, and the local $F$-liftings are compatible with $\tilde f$.
Let  $(V, \nabla)$ be a vector bundle on $X$ with an integrable $\lambda$-connection for some $\lambda\in k$. 
Let us set $n=\dim X$,  $n'=\dim Y$ and let $L=(L_1,...,L_{n-1})$ be a collection of ample line bundles on $\bar X$.
Assume that  $(V, \nabla)$ is  slope $L$-semistable of rank $r\le p$ with
$\int_{\bar X} \Delta (j_*V)= 0$. Then $f^*(V, \nabla)$ is slope $M$-semistable for any collection $M=(M_1,...,M_{n'-1})$ of nef line bundles  on $\bar Y$.
\end{Theorem} 
 
\begin{proof}
Theorem \ref{reflexive-loc-free-Delta-Higgs} and \cite[Theorem 5.6]{La-Bog-normal} imply that there exists a canonical Higgs--de Rham flow of $(V, \nabla)$, i.e.,  a Higgs--de Rham flow   
$${  \xymatrix{
	(V_{-1}, \nabla _{-1}) \ar[rd]^{\Gr _{S_{-1}}}	&& (V_0, \nabla _0)\ar[rd]^{\Gr _{S_0}}&& (V_1, \nabla _1)\ar[rd]^{\Gr _{S_1}}&\\
	&	(\cE_0, \theta _0)\ar[ru]^{C_{\tilde X}^{-1}}&&	(\cE_1, \theta_1)\ar[ru]^{C_{\tilde X}^{-1}}&&...\\
}}$$
in which $S_{i}$ are Simpson's filtrations. In particular, all $(\cE_m, \theta _m)$ and $(V_m, \nabla _m)$ are slope $L$-semistable. Moreover, Theorem \ref{reflexive-loc-free-Delta-Higgs} and \cite[Corollary 5.11]{La-Bog-normal} imply that for all $m\ge 0$ we have $\int_{\bar X} \Delta (j_*\cE_m) = 0$. Let us write $p^m=q_mr+r_m$ for some integers $q_m$ and $0\le r_m<r$.
Let us also set $\cE_m':=\cE_m (-q_mc_1(\cE_0))$. Then Theorem \ref{boundedness} implies that the sets $\{  j_*\cE_m'\}_{m\ge 0}$ and $\{  \cE_m'\}_{m\ge 0}$ are bounded. Since all $\cE'_m$ are locally free, the set  $\{  f^*\cE_m'\}_{m\ge 0}$ is also bounded. Our assumptions imply that the pullback of the above canonical Higgs--de Rham flow  of $(V, \nabla)$ gives a  Higgs--de Rham flow  of $f^*(V, \nabla)$. Now the proof is the same as that of (1) in Theorem \ref{reflexive-loc-free-Delta-Higgs}. Assume that $(W,\nabla_W)\subset f^* (V,\nabla)$ is $M$-destabilizing. 
We set $(\cF_0,\eta _0):=\Gr_{f^*S_{-1}}(W,\nabla_W)$ and consider the sequence 
	$$(\cF_m,\eta_m):= \Gr _{f^*S_{m-1}} C^{-1}_{\tilde Y}(\cF_{m-1}, \eta_{m-1}).$$
Then  $\cF_m \otimes f^*(\det \cE_0)^{-q_m}$ give  subsheaves of $f^*\cE_m'$, whose $M$-slopes grow to infinity. This gives  a contradiction with Proposition \ref{needed-for-independence-of-compactification} as in a bounded family of sheaves on a normal projective variety, the slopes of subsheaves are uniformly bounded from the above.  	
\end{proof}

\medskip

\begin{Remark}\label{problems-with-vanishing-discriminant}
It is possible that in the above theorem we have 
$$\int_{\bar Y}\Delta (j'_*(f^*V))=0.$$
In case $X=\bar X$ and $Y=\bar Y$ this holds by \cite[Theorem 2.2]{La-JEMS}
as $\Delta (f^*V)=f^*\Delta (V)$ vanishes in the group $B^2(Y)$ of algebraic equivalence classes of codimension $2$ cycles (equalities of Chern classes in \cite{La-JEMS} are asserted in the $\ell$-adic cohomology, but the proof shows them also in the ring $B^{*}(X)$ of algebraic equivalence classes).    
Since Chern classes of reflexive sheaves are not functorial, the proof from the smooth case does not work.
In principle, to show vanishing of the discriminant one can try to use  analogues of the boundedness argument in the proof of Corollary \ref{characterization-proj-flat}. Unfortunately,  
as the following example shows, this method breaks down even if $\bar Y$ is a smooth projective surface. In case of complex varieties, when small compactifications  have klt singularities, the above vanishing of the discriminant follows from \cite[Theorem 1.2]{GKPT3}  (and Theorem \ref{char0-loc-free-Higgs} to get vanishing of the second Chern form and not only its value on $H^{n-2}$ for ample $H$). The proof in this case uses pullbacks to maximally quasi--\'etale covers and \cite[Theorem 3.10]{GKPT1}. 
\end{Remark}
 
\begin{Example}
Let us consider $\cF:=\cE|_X$ on $j: X=\PP^2\backslash Z\hookrightarrow \bar X=\PP^2$ from Example \ref{example-large-num-flat}. In particular, we have $\int_{\bar X}\Delta (j_*\cF)=\int_{\bar X}\Delta (\cE)>0$. The Higgs bundle $(\cE, 0)$ has a canonical Higgs--de Rham flow given by 
$${  \xymatrix{
		(\cE, 0) \ar[rd]	&& (F^*_{\PP^2}\cE, \nabla _{can})\ar[rd]&& ((F^2_{\PP^2})^*\cE, \nabla _{can})\ar[rd]&\\
		&	(\cE, 0)\ar[ru]^{C_{\PP^2_{W_2(k)}}^{-1}}&&	(F^*_{\PP^2}\cE, 0)\ar[ru]^{C_{\PP^2_{W_2(k)}}^{-1}}&&...\\
}}$$
In this Higgs--de Rham flow all bundles are semistable but the set  of associated Higgs bundles $\{(F^m_{\PP^2})^*\cE\}_{m\ge 0}$ is not bounded as $\int_{\bar X}\Delta ((F^m_{\PP^2})^*\cE)$
grows to infinity with $m$. Clearly, this Higgs--de Rham flow restricts to a Higgs--de Rham flow of $(\cF, 0)$ and by Proposition \ref{needed-for-independence-of-compactification},  the set  of associated Higgs bundles in this restricted flow is also unbounded. However, $(\cF, 0)$ has another Higgs--de Rham flow, where in the first step we take the obvious  filtration with the associated graded $\cO_X\oplus \cO_X$ and then take a direct sum of Higgs--de Rham flows of the summands. In this Higgs--de Rham flow all bundles are semistable and  the set of associated Higgs bundles is bounded.
\end{Example}

\medskip

\begin{Remark}
	Using Theorem \ref{local-freeness-log-Higgs-case} one can generalize
	Theorem \ref{restriction-Higgs-bundles} to the logarithmic case. We leave formulation of the corresponding result to the interested reader.
\end{Remark}

\subsection{Local freeness in characteristic zero}

In this subsection  $k$ is an algebraically closed field of characteristic zero and $X$ is a normal projective $k$-variety of dimension $n$ with only quotient singularities in codimension $2$. We also fix a collection $L=(L_1,...,L_{n-1})$ of ample line bundles on $X$.

\begin{Theorem} \label{char0-loc-free-Higgs}
	 Let  $(V, \nabla)$ be a reflexive $\cO_X$-module with an integrable $\lambda$-connection for some $\lambda\in k$. Assume that $(V, \nabla)$ is  slope $L$-semistable with $$\int_X \ch_1 (V) L_1...L_{n-1}=0$$ and
	$$\int_X \ch_2 (V) L_1...\widehat{L_i}...L_{n-1}\ge 0$$
	for $i=1,...,n-1$. Then the following conditions are satisfied:
	\begin{enumerate}
		\item $[c_1(V)]=0\in \bar B^1(X)$ and $\int_X \ch_2 (V) =0$,
		\item if $M=(M_1,..., M_{n-1})$ is a collection of nef line bundles on $X$ then $(V, \nabla)$ is slope $M$-semistable,
		\item $V|_{X_{reg}}$ is locally free.
	\end{enumerate}
	Moreover, if $N^{\bullet }$ is a Griffiths transverse filtration on $(V, \nabla)$  such that  $(\cE:=\Gr _{N} V, \theta )$ is  slope $L$-semistable then the following conditions are satisfied:
	\begin{enumerate}
		\item $[c_1(\cE)]=0\in \bar B^1(X)$ and $\int_X \ch_2 (\cE^{**})= 0$,
		\item if $M=(M_1,..., M_{n-1})$ is a collection of nef line bundles on $X$ then $(\cE , \theta)$ is slope $M$-semistable,
		\item $\cE|_{X_{reg}}$ is locally free.
	\end{enumerate}
\end{Theorem}

\begin{proof}
The required assertions follow from Theorem \ref{reflexive-loc-free-Delta-Higgs} by standard spreading out arguments analogous to that from the proof of \cite[Theorem 7.2]{La-Bog-normal}. One should only remark that by \cite[Remark 3.18]{La-Bog-normal}, the condition $[c_1(V)]=0\in \bar B^1(X)$ can be realized numerically,
so it follows from the one on reductions modulo $p$. Another remark is that it is sufficient to check  (2) for collections of ample line bundles (note that nefness is not an open condition, so we cannot directly use the spreading out argument). 
\end{proof}

\begin{Remark}
\begin{enumerate}
	\item 	Using Theorem \ref{local-freeness-log-Higgs-case} instead of Theorem \ref{reflexive-loc-free-Delta-Higgs}, one also gets an analogous result in the logarithmic case. We leave a precise formulation of the result to the reader.
	\item 	If $\lambda \ne 0$ then  local freeness of  $V|_{X_{reg}}$  is well-known (see, e.g., \cite[4.5.1]{ABC}). However, this is no longer so if $\lambda=0$ or in the logarithmic case.
	\item If $X$ has only klt singularities (so in particular quotient singularities in codimension $2$) and $\lambda=0$ then the first part of Theorem \ref{char0-loc-free-Higgs} follows from \cite[Theorem 1.2]{GKPT3}. However, even in this case the second part of Theorem \ref{char0-loc-free-Higgs} does not  easily follow  from the results of \cite{GKPT3}.
\end{enumerate}	
\end{Remark}

In some cases assumptions of the above theorem are automatically satisfied due to the following lemma. 

\begin{Lemma}
Let  $(V, \nabla)$ be a reflexive $\cO_X$-module with an integrable $\lambda$-connection for some $\lambda\in k^*$. Then $(V, \nabla)$ is slope $L$-semistable with $$\int_X \ch_1 (V) L_1...L_{n-1}=0$$ and
	$$\int_X \ch_2 (V) L_1...\widehat{L_i}...L_{n-1}=0$$
	for $i=1,...,n-1$.
\end{Lemma}

\begin{proof}
Without loss of generality we  can assume that $\lambda=1$ and all $L_i$ are very ample.
If $C \in |L_1|\cap ...\cap |L_{n-1}|$ is a general complete intersection curve then $(V,\nabla)|_C$ is a vector bundle with an integrable connection on  a smooth projective curve. So $\deg V=0$ and $(V,\nabla)|_C$ is semistable as any vector subbundle with an integrable connection also has degree $0$. This implies the first two assertions. To prove the last one, we can restrict to a (very) 
general complete intersection surface $S \in |L_1|\cap... \cap |\widehat{L_i}|\cap  ...\cap |L_{n-1}|$. If $S$ is smooth then the assertion is clear as rational Chern classes of a vector subbundle with an integrable connection vanish.
In general, one can use \cite[Lemma 7.1]{La-Bog-normal} to reduce to this case. We omit the details as in this case the result follows also from \cite[Lemma 3.17]{GKPT1}. 
\end{proof}

Applying the above theorem to $(V, \nabla)=(\cE, 0)$ and $\lambda=0$ we get the following corollary.

\begin{Corollary}
Let $H$ be an ample line bundle on $X$ and let
$\cE$ be a coherent reflexive $\cO_X$-module such that 
$$\int _X \ch_1(\cE) H^{n-1}=\int _X\ch_2(\cE) H^{n-2} =0.$$
If $\cE$ is slope $H$-semistable then $\cE|_{X_{reg}}$ is a  numerically flat vector bundle, $[c_1(\cE)]=0\in \bar B^1(X)$ and $\int_X \ch_2 (\cE)= 0$.
\end{Corollary}

It is tempting to say that  $\cE|_{X_{reg}}$ in the above corollary is strongly numerically flat but this is not so easy to check, unless $X$ has only quotient singularities when it can be reduced to the smooth case. The problem is to check vanishing of the second Chern character after restricting to surfaces (cf. Remark \ref{problems-with-vanishing-discriminant} for an even easier problem).
If $X$ is a complex projective variety with klt singularities, this can be proven similarly to implication
$(1.2.2)\Rightarrow (1.2.1)$ in  \cite[Theorem 1.2]{GKPT3}. 

\medskip

The above corollary and  \cite[Theorem 1.1]{HP} imply the following non-trivial result.

\begin{Corollary} 
Let $\cE$ be a coherent reflexive $\cO_X$-module such that  $$c_1(\cE) H^{n-1}=0.$$ Assume that all reflexive symmetric powers $\Sym ^{[m]}\cE$ are slope $H$-stable and $\cE$ is pseudoeffective. If $X$ is smooth in codimension $2$ then
  $\cE|_{X_{reg}}$ is a numerically flat vector bundle. In particular, $\cE|_{X_{reg}}$ is locally free.
\end{Corollary} 

Since  $X$ does not need to have klt singularities, the above corollary does not follow from \cite{GKPT3}.

\section*{Acknowledgements}

The author would like to thank Daniel Greb for answering some questions related to \cite{GKPT3}. 
The author would also like to thank the referee for the remarks. The final check of the revision of the paper that removed some small errors was done using ChatGPT (GPT-6 Astra).

The  author was partially supported by Polish National Science Centre (NCN) contract number 2021/41/B/ST1/03741.

%\footnotesize

\end{document}